\documentclass[12pt]{amsart}
\usepackage{amsfonts}
\usepackage[utf8]{inputenc}
\usepackage{amssymb,amsmath,amsthm,amsrefs,newlfont,enumerate,array,graphicx,tikz,tkz-tab,xcolor,bbold,float,subcaption}
\usepackage{thmtools,thm-restate}
\usepackage{mathrsfs}
\usepackage{graphicx}
\usepackage{a4wide}
\usepackage[colorlinks, citecolor=blue, linkcolor=red, pdfstartview=FitB]{hyperref}
\usepackage[shortlabels]{enumitem}

\title{Weighted Composition operators on de Branges--Rovnyak spaces}

\author[Fricain]{Emmanuel Fricain}
 \address{Univ. Lille, CNRS, UMR 8524 - Laboratoire Paul Painlevé, F-59000 Lille, France}
 \email{emmanuel.fricain@univ-lille.fr}

\author[Karaki]{Muath Karaki}
\address{Department of Mathematics, An--Najah National University, Nablus, Palestine}
\email{muath.karaki@najah.edu}

\author[Mashreghi]{Javad Mashreghi}
\address{D\'epartement de math\'ematiques et de statistique,
         Universit\'e Laval,
         Qu\'ebec, QC,
         Canada G1K 7P4}
\email{javad.mashreghi@mat.ulaval.ca}

\author[Ostermann]{Ma\"eva Ostermann}
\address{Univ. Lille, CNRS, UMR 8524 - Laboratoire Paul Painlevé, F-59000 Lille, France}
\email{maeva.ostermann@univ-lille.fr}



\theoremstyle{plain}
 \usepackage[noabbrev]{cleveref}
\makeatletter
\newcommand{\customlabel}[2]{\def\@currentlabel{#2}\label{#1}}
\makeatother
\newtheorem{theorem}{Theorem}[section]
\newtheorem*{theorem*}{Theorem}
\newtheorem{lemma}[theorem]{Lemma}
\newtheorem{corollary}[theorem]{Corollary}

\newtheorem{proposition}[theorem]{Proposition}

\newtheorem{question}[theorem]{Question}
\theoremstyle{definition}

\newtheorem{remark}[theorem]{Remark}
\newtheorem{example}[theorem]{Example}

\makeatletter
\let\c@equation\c@theorem

\makeatother
\renewcommand{\phi}{\varphi}
\def\HH{\mathscr{H}}
\def\P{\mathscr{P}}
\def\W{W_{u,\phi}}
 \DeclareMathOperator*{\supess}{sup\,ess}

\begin{document}
\begin{abstract}
In this paper, we characterize the boundedness and the compactness of weighted composition operators acting on a de Branges--Rovnyak space $\mathcal{H}(b)$, where the symbol $b$ is a rational function in the unit ball of $H^\infty$ that is not a finite Blaschke product. Our results extend those of \cite{alhajj2022composition} by exploiting a close relationship between weighted composition operators on $\mathcal{H}(b)$ and their counterparts on the Hardy space $H^2$.
\end{abstract}

\maketitle

\section{Introduction}
An important example of operators that has been extensively studied on Hilbert spaces of analytic functions on the open unit disk $\mathbb{D}$ of the complex plane is the class of composition operators
\[
C_\varphi:f\longmapsto f\circ\varphi,
\]
where $\varphi$ is an analytic self-map of $\mathbb{D}$, called the symbol of $C_\varphi$. In particular, thanks to Littlewood’s subordination principle, it is well known that $C_\varphi$ is always bounded on the Hardy space $H^2$ of the open unit disc. A natural extension of this class is the family of weighted composition operators
\[
W_{u,\varphi}:f\longmapsto u\cdot f\circ \varphi,
\]
where $u$ is some analytic function, called the weight of $W_{u,\varphi}$. In other words, a weighted composition operator is the combination of a composition operator and a multiplication operator.

Weighted composition operators are widely studied in spaces such as $H^2$, the Bergman space $A^2$, and various spaces of analytic or harmonic functions. They exhibit richer behavior than ordinary composition operators, as the presence of a weight can significantly influence boundedness, compactness, spectral properties, and dynamical features of the mapping. Moreover, the weighted composition operators appear naturally in the study of linear isometries. See for instance \cite{MR4762036}. A characterization of boundedness of weighted composition operators on $H^2$ has been obtained in terms of Carleson measures \cites{CONTRERAS2001224,Gallardo-Gutierrez-2010}. Notably, it is possible to find examples where the weight $u$ is unbounded (so that the corresponding multiplication operator is not bounded on $H^2$), yet the choice of the symbol $\varphi$ ensures that the weighted composition operator $W_{u,\varphi}$ remains bounded on $H^2$.

In this paper, we are studying the weighted composition operators on a particular class of de Branges--Rovnyak spaces $\mathcal H(b)$. More precisely, we focus on the case when $b$ is a rational (but, not finite Blaschke product) function in the unit ball of $H^\infty$, with $\|b\|_\infty=1$. It is known that $b$ has a Pythogorean mate $\tilde a$, that is $\tilde a$ is the unique outer function such that $|\tilde a|+|b|=1$ a.e. on $\mathbb T=\partial\mathbb D$ and $\tilde a(0)>0$. Indeed, in this case the Pythogorean mate can be obtained by the Fej\'er--Riesz theorem. Hence, in particular, it is also a rational function without zeros in $\mathbb D$ and without poles in $\overline{\mathbb D}$. Since $\|b\|_\infty=1$, the function $\tilde a$ has some zeros $\zeta_1,\zeta_2,\dots,\zeta_n$ on $\mathbb T$ with respective multiplicities $m_1,\dots,m_n$. Moreover, it is known \cites{MR3110499,MR3503356} that the space $\mathcal H(b)$ can be decomposed as
\begin{equation}\label{eq:decomposition-Hb}
    \HH(b)=\prod_{i=1}^{n}(z-\zeta_i)^{m_i}H^2\oplus \P_{N-1},
\end{equation} where $N=\sum_{i=1}^nm_i$ and $\P_{N-1}$ is the set of polynomials of degree less or equal to $N-1$. Note that the sum in \eqref{eq:decomposition-Hb} is a direct sum. It turns out that de Branges--Rovnyak spaces, in this particular case, coincide with range spaces of co-analytic Toeplitz operators. More precisely, if we denote by $a=\prod_{i=1}^na_i,\ a_i(z)=(z-\zeta_i)^{m_i},$  it is known that 
\begin{equation}\label{eq:Hb-coincide-avecMabar}
\HH(b)=\mathcal{M}(\overline{\tilde a})=\mathcal{M}(\overline{a}),
\end{equation}
where $\mathcal{M}(v)$ denotes the range of the Toeplitz operator $T_v$ with the symbol $v\in L^\infty(\mathbb T)$, equipped with the range norm, making $T_v$ a surjective partial isometry from $H^2$ onto $\mathcal M(v)$. At the same time, \eqref{eq:decomposition-Hb} can be rewritten as $\mathcal H(b)=\mathcal M(a)\oplus \P_{N-1}$. We will equip the space $\HH(b)$ with the following norm 
\begin{equation}\label{eq:norm-Hb-decomposition}
\|f\|_b^2:=\|g\|^2_2+\|p\|^2_2,\qquad f=ag+p, \, (g,p)\in H^2\times \P_{N-1}.
\end{equation}
This norm is not the usual norm in the de Branges--Rovnyak spaces, but it is an equivalent norm in this particular case. 

An important property of functions $f$ in $\HH(b)$, which easily follows from \eqref{eq:decomposition-Hb}, is that, for every $1\leq i\leq n$, the function $f$ and its derivatives up to order $m_i-1$ have non-tangential limits at points $\zeta_i$. Furthermore, if $f=ag+p$, with $g\in H^2$ and $p\in\P_{N-1}$, then the polynomial $p$ satisfies the following interpolation problem:
\begin{equation}\label{eq:partie-polynomial-interpolation}
p^{(k)}(\zeta_i)=f^{(k)}(\zeta_i), \qquad\text{for every }1\leq i\leq n,\,0\leq k\leq m_i-1.
\end{equation}
In particular, 
\begin{equation}\label{eq:description-Ma-avec-derivee}
\mathcal M(a)=\left\{f\in\HH(b):f^{(k)}(\zeta_i)=0\text{ for every }1\leq i\leq n,\,0\leq k\leq m_i-1\right\}.
\end{equation}
We can also easily see from \eqref{eq:decomposition-Hb} that $\HH(b)$ contains the set of polynomials. It can be proved that this set is dense in $\HH(b)$. See \cite{Fricain-2015-vol2} for a detailed account on $\HH(b)$ spaces and more precisely Section 27.5 therein for the particular case of rational $b$'s.

In \cite{alhajj2022composition}, the authors completely characterized  the symbols $\varphi$ such that the associated composition operator $C_\varphi$ is bounded/compact from $\mathcal H(b)$ into itself,  when $b$ is a rational function (not a finite Blaschke product) in the unit ball of $H^\infty$. The main point in the results obtained in \cite{alhajj2022composition} is that, using \eqref{eq:decomposition-Hb}, a genuine link can be made between the properties of $C_\varphi$ on $\mathcal H(b)$ and the properties of a related weighted composition operator $W_{u,\varphi}$ on $H^2$ for some appropriate weight $u$. On the other hand, it should also be mentioned that in \cite{MR3411049}, the authors presented a complete characterization for a composition operator $C_\varphi:\mathcal H(\Theta)\longrightarrow H^2$ to be compact when $\Theta$ is an inner function. Note that in this case, the space $\mathcal H(\Theta)$ coincides with the so-called model space $K_\Theta=(\Theta H^2)^\perp$, and it is of finite dimension if and only if $\Theta$ is a finite Blaschke product. In \cite{MR3438324}, the membership to Schatten classes for $C_\varphi:\mathcal H(\Theta)\longrightarrow H^2$ is studied, in particular, when the inner function $\Theta$ is a one-component inner function. Some generalizations of \cite{MR3411049} have been obtained in \cite{MR3915413}, where the authors study the compactness of $C_\varphi:\mathcal H(b)\longrightarrow H^2$ for some arbitrary function $b$ in the closed unit ball of $H^\infty$. Since $\mathcal H(b)$ is contractively contained in $H^2$, Littlewood's subordination principle makes it obvious that  the operator $C_\varphi:\mathcal H(b)\longrightarrow H^2$ is always bounded. The situation of \cite{alhajj2022composition} is different because it is unclear that when a composition operator $C_\varphi$ maps $\mathcal H(b)$ to itself; there are somehow some algebraic restrictions.

The goal of this paper is to pursue the line of research of \cite{alhajj2022composition} and study when a weighted composition operator $\W$ is bounded/compact from $\HH(b)$ into itself. Let us first remark that the operator $\W$ is bounded from $\HH(b)$ to itself if and only if $u\cdot f\circ \phi\in\HH(b)$ for all $f\in\HH(b)$. This is a simple consequence of the closed graph theorem and the fact that convergence in $\mathcal H(b)$ implies pointwise convergence. Moreover, if $\W$ is bounded on $\HH(b)$ then $u$ and $u\phi$ belong to $\HH(b).$ Indeed, since $1$ and $z\in\HH(b)$, we should have $u=u\cdot1\circ \phi\in\HH(b)$ and $ u\phi=u\cdot z\circ \phi\in\HH(b)$. 

So, without loss of generality, \textit{we assume that}
\begin{enumerate}[(H1)]
\item\label{H1} $u\in\HH(b)$, and
\item\label{H2} $u\phi\in\HH(b)$.
\end{enumerate}
In particular, for every $1\leq i\leq n,$ the function $u$ and its derivatives up to order $m_i-1$ have non-tangential limits at $\zeta_i$. In \Cref{Lemma:phi-in-M-ai-li}, we will prove the existence of non-tangential limits of $\phi$  at the points $\zeta_i$, when one of the derivatives of $u$ of order less or equal to $m_i-1$ does not vanish. Up to rearranging the sequence $\zeta_i$, $1\leq i\leq n$, we may assume that 
$\{\zeta_i:\ \exists \, 0\leq \ell\leq m_i-1,  u^{(\ell)}(\zeta_i)\ne 0 \text{ and } \phi(\zeta_i)\in\mathbb{D}  \}=\{\zeta_{p+1},\cdots,\zeta_n\}$. We will see that the continuity of $W_{u,\phi}$ on $\HH(b)$ is equivalent to the continuity of $W_{w,\phi}$ on $H^2$, with the following weight 
\begin{equation}
    \label{eq:defining-w}
w=\frac{u\cdot a\circ\phi\cdot\prod_{j=p+1}^n\left(\phi-\phi(\zeta_j)\right)^{m_j}}{a},
\end{equation}
which is well defined as soon as \ref{H1} and \ref{H2} are satisfied. 
\par\smallskip

The main result of this paper is the following.
\begin{theorem}\label{thm10}
Let $u$ and $\varphi$ satisfy \ref{H1} and \ref{H2}.  Let $w$ defined by \eqref{eq:defining-w}. Then the following are equivalent:
    \begin{enumerate}[(i)]
        \item The operator $\W$ is bounded on $\HH(b)$;
        \item The operator  $W_{w,\phi}$ is bounded on $H^2$;
        \item  We have \[\sup_{\lambda\in \mathbb{D}} \int_{\mathbb{T}}(1-|\lambda|^2) \frac{|w(\zeta)|^2}{|1-\overline{\lambda}\phi(\zeta)|^2}dm(\zeta)<+\infty.\]
    \end{enumerate}
\end{theorem}

%whenever $W_{u,\phi}$ is bounded on $\HH(b)$, then either $u^{(\ell)}(\zeta_i)=0$ for all $0\leq \ell\leq m_i-1$, or there exists an $\ell$ with $0\leq \ell\leq m_i-1$ such that $u^{(\ell)}(\zeta_i)\not=0$ and $\phi(\zeta_i)\in \mathbb{D}\cup Z({a})$, where $Z({a})=\{\zeta_1,\dots,\zeta_n\}$.

%Therefore, without loss of generality, \textit{we may also assume that}

Note that, given any symbol $\phi:\mathbb D\to\mathbb D$ analytic, there is always a weight $u$ such that the operator $W_{u,\phi}$ is bounded on $\HH(b)$. For example, it is enough to take $u = ag$ with $g \in H^\infty$. Indeed, by the Littlewood subordination principle, for every function $f \in \HH(b)$, the function $a\cdot g\cdot f \circ \phi$ belongs to $aH^2 \subset \HH(b)$, which proves, by the closed graph theorem, the boundedness of $W_{u,\phi}$ on $\HH(b)$. This can also be seen using \Cref{thm10}, since in this case, the functions $u$ and $u\varphi$ belong to $\mathcal M(a)$, hence to $\HH(b)$ and the associated weight $w$ belongs  to $H^\infty$. As we noted as part of the general philosophy, this shows that the presence of a weight improves the properties of the associated operator $W_{u,\phi}$.
\par\medskip

This paper is organized as follows. In \Cref{section:decomposition}, we will give an interesting split 
of $\mathcal M(\overline{a})$ when $a$ can be factorized as the 
product of finitely many $a_j$, $1\leq j\leq n$, which satisfy a 
corona condition. In \Cref{section:multiplicative}, we will use this result to prove 
that under conditions \ref{H1} and \ref{H2}, the function 
$u\cdot p(\phi)$ belongs to $\HH(b)$ for every polynomial $p$.  In Section~\ref{section:limites}, we will also show that the boundedness of $W_{u,\phi}$ on $\HH(b)$ may impose some restrictions on the non-tangential limits of $\varphi$ at points $\zeta_i$ that we will denote by \ref{H3}. In this section, we will also see that if $u$ and $\varphi$ satisfy \ref{H1} and \ref{H2} and if $w\in H^2$, then we necessarily have \ref{H3}.  \Cref{section:proof} is devoted to the proof of \Cref{thm10}, while, in \Cref{section:csq}, we give some examples and some simpler sufficient conditions for the boundedness of $W_{u,\varphi}$ on $\HH(b)$. We also examine the particular case where the symbol $\varphi$ has non-unimodular non-tangential limits at every point $\zeta_i$, as well as some interesting connections with the angular Carath\'eodory limits of $\varphi$. Section 7 is devoted to the study of admissible weights, that is, the weights $u$ that give rise to  bounded operators $\W$ on $\HH(b)$ for a fixed symbol $\phi$. Finally, the last section is devoted to the study of the compactness of the operator $W_{u,\varphi}$ on $\HH(b)$. 

\section{Decomposition of $\mathcal M(\overline{a})$ spaces}\label{section:decomposition}
Recall that if $\varphi,\psi\in H^\infty$, we have 
\begin{equation}\label{eq:relation-commutation-Toeplitz}
T_{\overline{\phi}}T_{\overline{\psi}}=T_{\overline{\psi}}T_{\overline{\phi}}=T_{\overline{\phi\psi}},
\end{equation}
which immediately gives 
\begin{equation}\label{eq11ZEE:inclusionMabar}
\mathcal M(\overline{\phi\psi})\subset \mathcal M(\overline{\phi})\cap \mathcal M(\overline{\psi}).
\end{equation}
Moreover, if $\varphi_{\text{i}}$ is an inner function, then $T_{\overline{\varphi_{\text{i}}}}H^2=H^2$, and then \eqref{eq:relation-commutation-Toeplitz} implies that $\mathcal M(\overline{\varphi})=\mathcal M(\overline{\varphi_{\text{o}}})$, where $\varphi_{\text{o}}$ is the outer part of $\varphi$. In other words, when one studies the properties of $\mathcal M(\overline{a})$, we may assume that $a$ is an outer function in $H^\infty$. Then the operator $T_{\overline{a}}$ is one-to-one, and hence an isometry from $H^2$ onto $\mathcal M(\overline{a})=T_{\overline{a}}H^2$  equipped with the range norm
\begin{equation}\label{eq:normMabar}
\|f\|_{\mathcal M(\overline{a})}=\|g\|_2,\qquad f=T_{\overline{a}}(g)\in\mathcal M(\overline{a}).    
\end{equation}
It follows easily from \eqref{eq:normMabar} that $\mathcal M(\overline{a})$ is boundedly contained in $H^2$, and thus is a reproducing kernel Hilbert space. See \cite{Fricain-2015-vol2}*{Section 17.2}.  

A natural question is to know when we have the equality in \eqref{eq11ZEE:inclusionMabar}. The following result, interesting in its own right, will be useful in our study of the boundedness of $W_{u,\phi}$ on $\HH(b)$. 
%Note that the implication $(i)\implies (ii)$ already appeared in \cite{MR4970563} in the case $n=2$. 

\begin{theorem}\label{thm-decomposition-spaces-Mabar}
Let $a_1,\dots,a_n\in H^\infty$ be outer functions. Then the following assertions are equivalent.
\begin{enumerate}[(i)]
    \item For every $1\leq i,j\leq n$ with $i\neq j$, we have $\inf_{\mathbb D}(|a_i|+|a_j|)>0$.
    \item $\mathcal M(\overline{a_1\cdots a_n})=\mathcal M(\overline{a_1})\cap\dots\cap \mathcal M(\overline{a_n})$.
\end{enumerate}
\end{theorem}
\begin{proof}
We first prove the implication $(i)\implies(ii)$. This implication was obtained by Hartmann and Lamberti in \cite{MR4970563}. For the convenience of the reader, we present the justification given in \cite{MR4970563}*{Proposition 3.7} for the case $n=2$, and then we explain how an induction argument can be used to establish the implication for all $n$.
\par\smallskip
-- \textit{Case $n=2$:} Suppose that $\inf_{\mathbb D}(|a_1|+|a_2|)>0$. By the Corona Theorem, there exist $h_1,h_2\in H^\infty$ such that $a_1h_1+a_2h_2=1$. 

Let $f\in \mathcal M(\overline{a_1})\cap\mathcal M(\overline{a_2})$, and  let $g_1,g_2\in H^2$ satisfy $f=T_{\overline{a_1}}g_1=T_{\overline{a_2}}g_2$. So we have
\begin{align*}
T_{\overline{a_1a_2}}\Big(T_{\overline{h_1}}g_2+T_{\overline{h_2}}g_1\Big)
&=T_{\overline{a_1h_1}}\Big(T_{\overline{a_2}}g_2\Big)+T_{\overline{a_2h_2}}\Big(T_{\overline{a_1}}g_2\Big)\\
&=T_{\overline{a_1h_1}}f+T_{\overline{a_2h_2}}f\\
&=T_{\overline{a_1h_1+a_2h_2}}f=f.
\end{align*}
Thus $f\in \mathcal M(\overline{a_1a_2})$, which proves the inclusion  $\mathcal M(\overline{a_1})\cap\mathcal M(\overline{a_2})\subset \mathcal M(\overline{a_1a_2})$. Since the other inclusion is always satisfied (see \Cref{eq11ZEE:inclusionMabar}), we conclude $\mathcal M(\overline{a_1})\cap\mathcal M(\overline{a_2})=\mathcal M(\overline{a_1a_2})$.
\par\smallskip
-- \textit{Induction:}  Assume now that for an integer $n\ge2$ the implication is satisfied for all $k$-tuples, with $2 \leq k \leq n$.

Let $a_1,\dots,a_{n+1}$ be outer functions satisfying $\|a_j\|_\infty\le 1$, and such that, for every $1\leq i,j\leq n$, $i\neq j$, $|a_i|+|a_j|\ge c$ for some constant $c>0$. Let $a=a_1\cdots a_n$. When we multiply all inequalities $|a_j|+|a_{n+1}|\ge c$, for $1\le j\le n$, and develop them, we obtain the linear combination of $2^{n}$ products of $|a_j|$ and only one without the term $|a_{n+1}|$, which is $|a|$. So we  have
\[2^n(|a|+|a_{n+1}|)\ge |a|+(2^n-1)|a_{n+1}|\ge\prod_{j=1}^n(|a_j|+|a_{n+1}|)\ge c^n.\]
Then $|a|+|a_{n+1}|\ge c'$ with $c'=(c/2)^n>0$. By the hypothesis of induction applied to the pair $(a,a_{n+1})$ and to the $n$-tuple $(a_1,\dots,a_n)$, we then deduce that
\begin{align*}
\mathcal M(\overline{a_1\dots a_{n+1}})=\mathcal M(\overline{a a_{n+1}})&=\mathcal M(\overline{a})\cap \mathcal M(\overline{a_{n+1}})\\
&=\mathcal M(\overline{a_1\dots a_{n}})\cap \mathcal M(\overline{a_{n+1}})=\mathcal M(\overline{a_1})\cap\dots\cap \mathcal M(\overline{a_{n+1}}).
\end{align*}
So, the implication is also true for $(n+1)$-tuples and thus, by induction, this is true for every tuple, and this concludes the proof of the implication $(i)\implies(ii)$.
\par\medskip
We now prove the implication $(ii)\implies(i)$. Regarding the first implication, we begin with the case $n=2$ and then deduce the general result.
\par\smallskip
-- \textit{Case $n=2$:} Suppose  that $\mathcal M(\overline{a_1})\cap\mathcal M(\overline{a_2})=\mathcal M(\overline{a_1a_2})$. We equip $\mathcal M(\overline{a_1a_2})$ with the norm
\[\|f\|=\sqrt{\|f\|_{\mathcal M(\overline{a_1})}^2+\|f\|_{\mathcal M(\overline{a_2})}^2}\,,\qquad f\in \mathcal M(\overline{a_1a_2}).\]
This norm is complete. Indeed, let $(f_n)\subset \mathcal M(\overline{a_1a_2})$ be a Cauchy sequence for $\|\cdot\|$. Then $(f_n)$ is also a Cauchy sequence in $\mathcal M(\overline{a_1})$ and in $\mathcal M(\overline{a_2})$. Since $\mathcal M(\overline{a_1})$ and $\mathcal M(\overline{a_2})$ are complete spaces, there exist $g_1\in \mathcal M(\overline{a_1})$ and $g_2 \in \mathcal M(\overline{a_2})$ such that 
\[\|f_n-g_1\|_{\mathcal M(\overline{a_1})}\longrightarrow 0~\text{and}~\|f_n-g_2\|_{\mathcal M(\overline{a_2})}\longrightarrow 0~\text{when}~n\to\infty.\]
But $\mathcal M(\overline{a_1})$ and $\mathcal M(\overline{a_2})$ are Reproducing Kernel Hilbert Spaces, so this implies the pointwise convergence of $(f_n)$ to $g_1$ and $g_2$ simultaneously. Hence, $f:=g_1=g_2\in \mathcal M(\overline{a_1})\cap\mathcal M(\overline{a_2})=\mathcal M(\overline{a_1a_2})$ and
\[\|f_n-f\|=\sqrt{\|f_n-g_1\|_{\mathcal M(\overline{a_1})}^2+\|f_n-g_2\|_{\mathcal M(\overline{a_2})}^2}\longrightarrow0~\text{when}~n\to\infty.\]
We thus prove that $\mathcal M(\overline{a_1a_2})$ is a Banach space equipped with the norm $\|\cdot\|$.

Now, let $f\in \mathcal M(\overline{a_1a_2})$, and let $g\in H^2$ satisfy $f=T_{\overline{a_1a_2}}g$. According to \eqref{eq:normMabar}, we have $\|f\|_{\mathcal M(\overline{a_1a_2})}=\|g\|_{H^2}$. Moreover, since $\|a_1\|_\infty\le 1$ and $\|a_2\|_\infty\le 1$, the operators $T_{\overline{a_1}}$ and $T_{\overline{a_2}}$ are contractions on $H^2$ and, using \eqref{eq:relation-commutation-Toeplitz} and \eqref{eq:normMabar},  we have
\[\|f\|_{\mathcal{M}(\overline{a_1})}=\|T_{\overline{a_2}}g\|_{H^2}\le \|T_{\overline{a_2}}\|\|g\|_{H^2}\le\|g\|_{H^2}~\text{and}~\|f\|_{\mathcal{M}(\overline{a_2})}\le \|g\|_{H^2}.\]
Therefore, we deduce that 
\[\|f\|=\sqrt{\|f\|_{\mathcal M(\overline{a_1})}^2+\|f\|_{\mathcal M(\overline{a_2})}^2}\le \sqrt2 \|f\|_{\mathcal M(\overline{a_1a_2})}.\]
Since $\mathcal M(\overline{a_1a_2})$ is complete for these two norms, by the Banach isomorphism theorem there exists a constant $c>0$ such that for all $f\in \mathcal M(\overline{a_1a_2})$, 
\begin{equation}\label{Eq1}
    \|f\|_{\mathcal M(\overline{a_1a_2})}\le c\|f\|=c\sqrt{\|f\|_{\mathcal M(\overline{a_1})}^2+\|f\|_{\mathcal M(\overline{a_2})}^2}.
\end{equation}
Recall now that, for all $z\in\mathbb D$ and every function $u\in H^\infty$, we have $T_{\overline u}k_z=\overline{u(z)}k_z$. Hence, if $u$ is outer, we get $\|k_z\|_{\mathcal M(\overline{u})}=\frac{\|k_z\|_{H^2}}{|u(z)|}$. Then applying \eqref{Eq1} to $f=k_z$ gives
\[\frac{\|k_z\|_{H^2}}{|a_1(z)a_2(z)|}\le c\sqrt{\frac{\|k_z\|_{H^2}^2}{|a_1(z)|^2}+\frac{\|k_z\|_{H^2}^2}{|a_2(z)|^2}},\]
and so
\[|a_1(z)|+|a_2(z)|\ge|a_1(z)|^2+|a_2(z)|^2\ge\frac1{c^2}.\]
This means that $a_1$ and $a_2$ satisfy $(i)$.
\par\smallskip
-- \textit{General case}: Suppose that  $\mathcal M(\overline{a_1\dots a_n})=\mathcal M(\overline{a_1})\cap\dots\cap \mathcal M(\overline{a_n})$.  Fix $1\le j\le n$ and let $a_j^\sharp$ be defined by $a_j^\sharp=\prod\limits_{k\neq j} a_k$. On the one hand, according to \eqref{eq11ZEE:inclusionMabar}, we have $\mathcal M\left(\overline{a_ja_j^\sharp}\right)\subset\mathcal M(\overline{a_j})\cap\mathcal M\left(\overline{a_j^\sharp}\right)$. On the other hand, 
 \[\mathcal M(\overline{a_j})\cap\mathcal M\left(\overline{a_j^\sharp}\right)
=\mathcal M(\overline{a_j})\cap\mathcal M\left(\overline{\prod_{k\neq j}a_k}\right)\subset \bigcap_{k=1}^n\mathcal M(\overline{a_k})=\mathcal M(\overline{a_1\dots a_n})=\mathcal M\left(\overline{a_ja_j^\sharp}\right).
\]
Then we deduce $\mathcal M\left(\overline{a_ja_j^\sharp}\right)=\mathcal M(\overline{a_j})\cap\mathcal M\left(\overline{a_j^\sharp}\right)$. But we already proved the implication $(ii)\implies(i)$ for pairs of outer functions, and thus there exists $c>0$ such that $|a_j^\sharp|+|a_j|\ge c$. Since for every $k\neq j$, we have $|a_k|\ge|a_j^\sharp|$, we deduce that $|a_k|+|a_j|\ge c$ and so $(i)$ is satisfied.
\end{proof}

\section{Multiplicative properties of $\HH(b)$ functions} \label{section:multiplicative}
We assume from now on that $b$ is a rational (not inner) function in the closed unit ball of $H^\infty$ with $\|b\|_\infty=1$, and the zeros of its pythagorean mate on $\mathbb T$ are $\zeta_1,\dots,\zeta_n$ with respective multiplicities $m_1,\dots,m_n$. Hence $\HH(b)$ is decomposed as in \eqref{eq:decomposition-Hb}. In this case, its set of multipliers $\mathfrak{M}(\HH(b))=\{\varphi\in \mbox{Hol}(\mathbb D):\varphi f\in\HH(b),\,\forall f\in\HH(b)\}$ is described as
\begin{equation}\label{eq:multiplier3434}
\mathfrak{M}(\HH(b))=\HH(b)\cap H^\infty.
\end{equation}
See \cite{MR3967886}*{Proposition 3.1}. 
 \begin{theorem}\label{thm5}
     Let $u$ and $\varphi$ satisfy \ref{H1} and \ref{H2}. Then, for every polynomial $p$, the function $u\cdot p\circ \phi$ belongs to $\HH(b)$.
 \end{theorem}
It should be noted that, if we assume furthermore that $\varphi\in\HH(b)$, then the proof of \Cref{thm5} is trivial, because, according to \eqref{eq:multiplier3434}, we deduce that $\varphi\in\mathfrak{M}(\HH(b))$, and then \ref{H1} and \ref{H2} immediately imply that $u\varphi^k\in \HH(b)$ for every $k\geq 0$. In the general case, the proof of \Cref{thm5} is more complex and we need to establish several technical lemmas. We start with the the following result.

 \begin{lemma}\label{Lemma:phi-in-M-ai-li}
     Let $u$ and $\varphi$ satisfy \ref{H1} and \ref{H2}. Assume also that there exist $1\leq i\leq n$ and $0\leq \ell_i\leq m_i-1 $ such that $u(\zeta_i)=\cdots=u^{(\ell_i-1)}(\zeta_i)=0$ and $u^{(\ell_i)}(\zeta_i)\not=0$. Then $\phi\in \mathcal{M}(\overline{{a_{i,\ell_i}}}),$ where $a_{i,\ell_i}(z)=(z-\zeta_i)^{m_i-\ell_i}.$ In particular, the function $\varphi$ has a non-tangential limit at $\zeta_i$.
\end{lemma}

\begin{proof}
We know from \eqref{eq:Hb-coincide-avecMabar} that  $\HH(b)=\mathcal M(\overline a)$, where $a=\prod_{j=1}^n a_j$ and $a_j(z)=(z-\zeta_j)^{m_j}$. Hence, \Cref{thm-decomposition-spaces-Mabar} implies that $u$ and $u\phi$ belong to $\mathcal{M}(\overline{{a_i}})$, for every $1\leq i\leq n.$ In particular, using \eqref{eq:partie-polynomial-interpolation}, we can write 
 \begin{equation}\label{u-representation}
 u(z)=(z-\zeta_i)^{m_i}u_i(z)+\sum_{k=\ell_i}^{m_i-1}\frac{u^{(k)}(\zeta_i)}{k!}(z-\zeta_i)^k,
 \end{equation}
for some $u_i\in H^2$. Multiplying by $\phi$, we get that
 \begin{equation*} 
 u(z)\phi(z)=(z-\zeta_i)^{m_i}u_i(z)\phi(z)+\sum_{k=\ell_i}^{m_i-1}\frac{u^{(k)}(\zeta_i)}{k!}(z-\zeta_i)^k\phi(z).
 \end{equation*}
 But since $u\phi\in\mathcal{M}(\overline{a_i})$, similarly there exist $v_i\in H^2$ and $p_1\in\P_{m_i-1}$
 such that $u(z)\phi(z)=(z-\zeta_i)^{m_i}v_i(z)+p_1(z)$. Now write $p_1(z)=(z-\zeta_i)^{\ell_i}\widetilde q_1(z)+r_1(z)$, where $\widetilde q_1, r_1$ are polynomials and ${\deg}(r_1)<\ell_i. $ We thus deduce 
  \begin{equation*} 
 (z-\zeta_i)^{m_i} u_i(z) \phi(z)+\sum_{k=\ell_i}^{m_i-1} \frac{u^{(k)}(\zeta_i)}{k!}(z-\zeta_i)^k\phi(z)=(z-\zeta_i)^{m_i}v_i(z) +(z-\zeta_i)^{\ell_i}\widetilde q_1(z)+r_1(z).
 \end{equation*}
 Dividing by $(z-\zeta_i)^{\ell_i}$, we see that $r_1/(z-\zeta_i)^{\ell_i}\in H^2$ and, since ${\deg}(r_1)<\ell_i $, necessarily $r_1=0$. Hence, we get
  \begin{equation*} 
 (z-\zeta_i)^{m_i} u_i(z)\phi(z) +\sum_{k=\ell_i}^{m_i-1} \frac{u^{(k)}(\zeta_i)}{k!}(z-\zeta_i)^k\phi(z)=(z-\zeta_i)^{m_i}v_i(z) +(z-\zeta_i)^{\ell_i}\widetilde q_1(z) .
 \end{equation*}
 Since $\ell_i\leq m_i-1$, we can rewrite this as
  \begin{equation*} 
  \frac{u^{(\ell_i)}(\zeta_i)}{\ell_i!}(z-\zeta_i)^{\ell_i}\phi(z)=(z-\zeta_i)^{\ell_{i+1}}w_1(z) +(z-\zeta_i)^{\ell_i}\widetilde q_1(z),
 \end{equation*}
 for some $w_1\in H^2$. Now remember that $u^{(\ell_i)}(\zeta_i)\not=0$, so we deduce that 
\[
\phi(z)=(z-\zeta_i)\phi_1(z)+q_1(z),
\]
where $\phi_1\in H^2$ and $q_1$ is a polynomial. By induction, we assume that, for some $1\leq \ell\leq m_i-\ell_i-1$, 
\begin{equation}\label{eq:induction-phi-limite-radiale23344}
\phi(z)=(z-\zeta_i)^\ell\phi_\ell(z)+q_\ell(z),
\end{equation}
where $\phi_\ell\in H^2$ and $q_\ell$ is a polynomial. Since $u\phi\in \mathcal{M}(\overline{a_i})$ and $q_\ell u\in \mathcal{M}(\overline{a_i})$ (because $q_\ell$ is a polynomial, whence a multiplier of $\mathcal{M}(\overline{a_i}))$, we get from \eqref{eq:induction-phi-limite-radiale23344} that the function $(z-\zeta_i)^\ell\phi_\ell u$ belongs to $\mathcal{M}(\overline{a_i})$. Using \eqref{u-representation}, we have 
\begin{equation*} 
 (z-\zeta_i)^{\ell}   \phi_\ell(z) u(z)=(z-\zeta_i)^{m_i}(z-\zeta_i)^\ell\phi_\ell(z) u_i(z) +\sum_{k=\ell_i}^{m_i-1} \frac{u^{(k)}(\zeta_i)}{k!}(z-\zeta_i)^{k+\ell}\phi_\ell(z).
 \end{equation*}
 But $v_\ell:=(z-\zeta_i)^\ell\phi_\ell u_i=(\phi-q_\ell)u_i\in H^2$ and we can write 
 \begin{equation*} 
 (z-\zeta_i)^{\ell}   \phi_\ell(z) u(z)=(z-\zeta_i)^{m_i} v_\ell(z) +\sum_{k=\ell_i}^{m_i-1} \frac{u^{(k)}(\zeta_i)}{k!}(z-\zeta_i)^{k+\ell}\phi_\ell(z).
 \end{equation*}
 Since $(z-\zeta_i)^\ell\phi_\ell u\in\mathcal{M}(\overline{a_i})$, we can also write $(z-\zeta_i)^\ell\phi_\ell u=(z-\zeta_i)^{m_i}w+q$, where $w\in H^2$  and $q\in\P_{m_i-1}$. Now, using the euclidean division, write  $q=(z-\zeta_i)^{\ell_i+\ell}p+r,$  for some polynomials $p$ and $r$ such that ${\deg}(r)<\ell_i+\ell$. Then we obtain
 \begin{equation*} 
 (z-\zeta_i)^{m_i} v_\ell (z)  +\sum_{k=\ell_i}^{m_i-1} \frac{u^{(k)}(\zeta_i)}{k!}(z-\zeta_i)^{k+\ell}\phi_\ell(z)=(z-\zeta_i)^{m_i}w(z) +(z-\zeta_i)^{\ell_i+\ell}p(z)+r(z) .
 \end{equation*}
Observe that $\ell_i+\ell\leq m_i-1$ and so dividing by $(z-\zeta_i)^{\ell_i+\ell}$ we see that $r/(z-\zeta_i)^{\ell_i+\ell}\in H^2$ and thus $r=0$. In particular, 
   \begin{equation*} 
  \frac{u^{(\ell_i)}(\zeta_i)}{\ell_i!}(z-\zeta_i)^{\ell_i+\ell}\phi_\ell(z)=(z-\zeta_i)^{\ell_{i }+\ell+1}w_2(z) +(z-\zeta_i)^{\ell_i+\ell}p(z),
 \end{equation*}
for $w_2\in H^2$. Dividing by $ \frac{u^{(\ell_i)}(\zeta_i)}{\ell_i!}(z-\zeta_i)^{\ell_i+\ell}$, we see that $ \phi_\ell=(z-\zeta_i)w_3 +p_2$, where $w_3\in H^2,$ and $p_2$ is a polynomial. Thus, with \eqref{eq:induction-phi-limite-radiale23344}, we obtain
 \begin{align*}
     \phi(z)&=(z-\zeta_i)^\ell\left( (z-\zeta_i)w_3(z) + p_2(z)\right)+q_\ell(z)\\
 &= (z-\zeta_i)^{\ell+1} \phi_{\ell+1}(z)  +q_{\ell+1}(z),
 \end{align*}
 where $\phi_{\ell+1}=w_3\in H^2$ and $q_{\ell+1}=(z-\zeta_i)^\ell p_2+q_\ell$ is a polynomial. By induction hypothesis, we thus deduce that we have a decomposition as in \eqref{eq:induction-phi-limite-radiale23344} for all $1\le \ell\le m_i-\ell_i$, and so, in particular,
 \begin{equation}\label{eq:34EZSDSF2V}
    \phi(z)=(z-\zeta_i)^{m_i-\ell_i}\phi_{m_i-\ell_i}(z)+q_{m_i-\ell_i}(z),
 \end{equation}
  for some $\phi_{m_i-\ell_i}\in H^2$ and some polynomial $ q_{m_i-\ell_i}$. Thus, $\phi\in\mathcal{M}(\overline{a_{i,\ell_i}})$. Now, the fact that $\phi$ has a non-tangential limit at $\zeta_i$ follows immediately from \eqref{eq:34EZSDSF2V} and the fact that $m_i-\ell_i\geq 1$, which concludes the proof of \Cref{Lemma:phi-in-M-ai-li}.
\end{proof}

Lemma \ref{Lemma:phi-in-M-ai-li} immediately implies the following result, which is interesting in its own right.

\begin{corollary}\label{cor:phi-appartient-a-Hb}
Let $u$ and $\phi$ such that $u,u\varphi\in\HH(b)$. 
Assume that for every $1\leq i\leq n$, $u(\zeta_i)\neq 0$. 
Then $\varphi$ itself belongs to $\HH(b)$.    
\end{corollary}
\begin{proof}
We can apply \Cref{{Lemma:phi-in-M-ai-li}} with $\ell_i=0$ for every $1\leq i\leq n$. In this case, $a_{i,\ell_i}=a_i$, and we deduce that $\phi\in\mathcal M(\overline{a_i})$ for every $1\leq i\leq n$. The concludion follows from \eqref{eq:Hb-coincide-avecMabar} and \Cref{thm-decomposition-spaces-Mabar}.
\end{proof}

\begin{remark}
Note that under the hypothesis of \Cref{cor:phi-appartient-a-Hb}, the proof of \Cref{thm5} is easy. Indeed, according to \Cref{cor:phi-appartient-a-Hb}, the function $\varphi$ belong to $\HH(b)\cap H^\infty$, whence is a multiplier according to \eqref{eq:multiplier3434}. Now, it is clear that if $u\in\HH(b)$, then we have $u\cdot p\circ\phi\in\HH(b)$.
\end{remark}

The second step to prove \Cref{thm5} is the following result.
 \begin{lemma}\label{Lemma4}
Let $a_i(z)=(z-\zeta_i)^{m_i}$ with $\zeta_i\in\mathbb T$ and $m_i\in\mathbb N$, let $a_{i,\ell_i}(z)=(z-\zeta_i)^{m_i-\ell_i}$ for some $0\leq\ell_i\leq m_i-1$,
let $v\in \mathcal{M}(\overline{a_i})$ such that $v(\zeta_i)= v^\prime(\zeta_i)=\cdots=v^{(\ell_i-1)}(\zeta_i)=0$ and let $\psi\in \mathcal{M}(\overline{a_{i,\ell_i}})\cap H^\infty$. Then $\psi v\in \mathcal{M}(\overline{a_i} )$.
 \end{lemma}

\begin{proof}
The case $\ell_i=0$ is straightforward. Indeed, in this case, we have $a_{i,\ell_i}=a_i$ and, according to \eqref{eq:multiplier3434}, the function $\psi$ belongs to $\mathcal M(\overline{a_i})\cap H^\infty=\mathfrak{M}(\mathcal M(\overline{a_i}))$. 
  
Assume now that $\ell_i\geq 1$.
According to \eqref{eq:decomposition-Hb} and \eqref{eq:partie-polynomial-interpolation}, there exists $v_1\in H^2$ such that 
\begin{align*}
v(z)=(z-\zeta_i)^{m_i}v_1(z)+\sum_{k=\ell_i}^{m_i-1}\frac{v^{(k)}(\zeta_i)}{k!}(z-\zeta_i)^k.
\end{align*}
Then
\begin{align*}
\psi(z) v(z)=(z-\zeta_i)^{m_i}v_1(z)\psi(z)+\sum_{k=\ell_i}^{m_i-1}\frac{v^{(k)}(\zeta_i)}{k!}(z-\zeta_i)^k\psi(z).      \end{align*}
Since $v_1\in H^2$ and $\psi\in H^\infty$, we have $(z-\zeta_i)^{m_i}v_1\psi\in(z-\zeta_i)^{m_i}H^2\subset\mathcal{M}(\overline{a_i})$. So it remains to prove that the function 
\[
h:= \sum_{k=\ell_i}^{m_i-1}\frac{v^{(k)}(\zeta_i)}{k!}(z-\zeta_i)^k\psi
\]
belongs to $\mathcal{M}(\overline{a_i})$. But, using that $\mathcal M(\overline{a_{i,\ell_i}})=a_{i,\ell_i} H^2\oplus \P_{m_i-\ell_i-1}$, we can write $\psi=(z-\zeta_i)^{m_i-\ell_i}\psi_1+q,$ where $\psi_1\in H^2$ and $q\in\P_{m_i-\ell_i-1}$. 
Hence,
\[h(z)=\underbrace{\sum_{k=\ell_i}^{m_i-1}\frac{v^{(k)}(\zeta_i)}{k!}(z-\zeta_i)^{m_i+k-\ell_i}\psi_1(z)}_{h_1(z)}+\underbrace{q(z)\sum_{k=\ell_i}^{m_i-1}\frac{v^{(k)}(\zeta_i)}{k!}(z-\zeta_i)^k}_{h_2(z)}. \] 
Since $h_2$ is a polynomial, $h_2\in\mathcal{M}(\overline{a_i})$ and 
\begin{align*}
h_1(z)&= (z-\zeta_i)^{m_i}\sum_{k=\ell_i}^{m_i-1}\frac{v^{(k)}(\zeta_i)}{k!}(z-\zeta_i)^{k-\ell_i}\psi_1(z)=  (z-\zeta_i)^{m_i}\psi_2(z),~\text{with}~ \psi_2\in H^2.
\end{align*}
Thus, $h_1\in (z-\zeta_i)^{m_i}H^2\subset \mathcal{M}(\overline{a_i}).$ Finally, we deduce that $h\in \mathcal{M}(\overline{a_i})$, which ends the proof of \Cref{Lemma4}.
 \end{proof}
 
We are now ready to prove \Cref{thm5}.
\begin{proof}[Proof of \Cref{thm5}]
Fix $1\leq i\leq n.$ Using \Cref{thm-decomposition-spaces-Mabar} and \eqref{eq:Hb-coincide-avecMabar}, it is sufficient to prove that 
\begin{equation}\label{eq:u-phi-k}
    u\phi^k\in \mathcal{M}(\overline{ {a_i}}),~\text{for every}~k\ge1.
\end{equation} 
First note that if $u^{(\ell)}(\zeta_i)=0$, for all $0\leq \ell\leq m_i-1,$ then \eqref{eq:description-Ma-avec-derivee} implies that $u\in \mathcal{M}({a _i})$. Since $\phi^k\in H^\infty$, we then deduce that $u\phi^k\in \mathcal{M}(a_i)\subset \mathcal{M}(\overline{a_i})$ for all $k\geq 1$. So we may assume that there is $0\leq \ell_i\leq m_i-1$ such that 
$u(\zeta_i)= u^\prime(\zeta_i)=\cdots=u^{(\ell_i-1)}(\zeta_i)=0$ and $u^{(\ell_i )}(\zeta_i)\not=0$. Then \Cref{Lemma:phi-in-M-ai-li} implies that $\phi\in  \mathcal{M}(\overline{ a_{i,\ell_i}})\cap H^\infty$. Recall that $\mathcal{M}(\overline{ a_{i,\ell_i}})\cap H^\infty=\mathfrak{M}(\mathcal M(\overline{a_{i,\ell_i}}))$ (see \eqref{eq:multiplier3434}), whence $\mathcal{M}(\overline{ a_{i,\ell_i}})\cap H^\infty$ is an algebra. Hence, for every $k\geq 1$, we have $\varphi^k\in \mathcal{M}(\overline{ a_{i,\ell_i}})\cap H^\infty$. It now follows from \Cref{Lemma4} that, for every $k\geq 1$, we have $u\varphi^k\in\mathcal M(\overline{a_i})$, which proves \eqref{eq:u-phi-k} and thus \Cref{thm5}.   
 \end{proof}

%We will prove  \eqref{eq:u-phi-k} by induction on $k\geq1$. For $k=1$ the assertion follows by assumptions and \Cref{thm-decomposition-spaces-Mabar}.
%
%Assume now that $v:=u\phi^k\in \mathcal{M}(\overline{ {a_i}})$ for some $k\geq1.$ We would like to apply  \Cref{Lemma4} and so we need to check that $v(\zeta_i)= v^\prime(\zeta_i)=\ldots=v^{(l_i-1)}(\zeta_i)=0$. Since $v\in\mathcal{M}(\overline{a_i})$ we can write $v$ as 
 %    \begin{equation}
  %       \label{eq:v-representation}
   %      v=(z-\zeta_i)^{m_i}v_1+
 %\sum_{k=0}^{m_i-1}\frac{v^{(k)}(\zeta_i)}{k!}(z-\zeta_i)^k ,    %\end{equation} where $v_1\in H^2,$ and on the other hand 
  %\begin{equation*}
%u=(z-\zeta_i)^{m_i}u_1+ \sum_{k=l_i}^{m_i-1}\frac{u^{(k)}(\zeta_i)}{k!}(z-\zeta_i)^k ,    \end{equation*}
%where 
 %\begin{equation*}        
 %v=u\phi^k=(z-\zeta_i)^{m_i}u_1\phi^k+
 %\sum_{k=l_i}^{m_i-1}\frac{u^{(k)}(\zeta_i)}{k!}(z-\zeta_i)^k \phi^k.  \end{equation*}
%In particular, since $\phi^k\in H^\infty,$ we get $v=O\left((z-\zeta)^{l_i}\right).$ Using  \eqref{eq:v-representation}, this implies that  $v(\zeta_i)= v^\prime(\zeta_i)=\ldots=v^{(l_i-1)}(\zeta_i)=0$. (Observe that $(z-\zeta_i)^{m_i}v_1=(z-\zeta_i)^{l_i} (z-\zeta_i)^{m_i-l_i}v_1=o( (z-\zeta_i)^{l_i}v_1), \ z\to\zeta_i$). Thus, we can apply  \Cref{Lemma4}
% which implies that $v\phi=u\phi^k\in\mathcal{M}(\overline{a_i}),$ which proves \eqref{eq:u-phi-k} and \Cref{thm5}.   
 %\end{proof}
\begin{remark}\label{remark6}
Let $1\leq \ell_i\leq m_i$ and let $u\in\mathcal M(\overline{a_i})$ which satisfy $u^{(\ell)}(\zeta_i)=0$ for every $0\leq \ell\leq \ell_{i}-1$. Let $f \in H^\infty$ be such that $v:=u\cdot f\in \mathcal M(\overline{a_i})$. Then $v$ and its derivatives up to order $\ell_i-1$ have non-tangential limits at $\zeta_i$ which satisfy $v^{(\ell)}(\zeta_i)=0$ for every $0\leq \ell\leq \ell_i-1$. 
\end{remark}
Indeed, the fact that $v$ and its derivatives up to order $\ell_i-1$ have non-tangential limits at $\zeta_i$ follows immediately from the fact that $v\in\mathcal M(\overline{a_i})$. Observe now that, by \eqref{eq11ZEE:inclusionMabar}, we have $u\in\mathcal M(\overline{a_i})\subset\mathcal M(\overline{(z-\zeta_i)^{\ell_i}})$. Taking account of the hypothesis and \eqref{eq:description-Ma-avec-derivee}, we deduce that $u(z)=(z-\zeta_i)^{\ell_i}u_1(z)$, where $u_1\in H^2$. Since $f\in H^\infty$, we thus get that $v(z)=(z-\zeta_i)^{\ell_i}v_1(z)$, where $v_1\in H^2$. Using one more time \eqref{eq:description-Ma-avec-derivee}, we then conclude that $v^{(\ell)}(\zeta_i)=0$ for every $0\leq \ell\leq \ell_i-1$. 
%The fact that $v$ and its derivatives up to order $\ell_i-1$ have non-tangential limits at $\zeta_i$ follows immediately from the fact that $v\in\mathcal M(\overline{a_i})$. 
%Now, since $f\in H^\infty$, we have $|v(z)|=|u(z)||f(\phi(z))|\leq |u(z)|\|f\|_\infty$. Letting $z\to\zeta_i$ non-tangentially, since $u(\zeta_i)=0$, we get that  $v(\zeta_i)=0.$ Now, write 
%\begin{equation}
%\left|\frac{v(z)-v(\zeta_i)}{z-\zeta_i}\right|=  \left|\frac{ u(z)}{z-\zeta_i}f(\phi(z))\right|\leq \left|\frac{ u(z)-u(\zeta_i)}{z-\zeta_i} \right|\|f\|_\infty.
%\end{equation}
%But $\dfrac{u(z)-u(\zeta_i)}{z-\zeta_i}$ tends to $u'(\zeta_i)=0$ as $z\to \zeta_i$ non-tangentially. Hence $v$ has a non-tangential derivative at $\zeta_i$ which satisfies $v'(\zeta_i)=0$.  The result now follows by induction on $\ell$ and \cite{Fricain-2015-vol2}*{Lemma 22.5}.

\begin{corollary}\label{cor:7}
Let $u$ and $\varphi$ satisfy \ref{H1} and \ref{H2}. Then for any $\lambda_j\in\mathbb{D}$ and $m_j\in\mathbb N$, $1\leq j\leq n$, and  for any polynomial $p$,  the function 
\begin{align*}
h:=\frac{u\cdot p\circ \phi}{~\prod_{j=1}^n(1-\overline{\lambda_j}\phi)^{m_j}~}
\end{align*}
belongs to $\HH(b)$.
 \end{corollary}
 
\begin{proof}
Fix $1\leq i\leq n$. According to \Cref{thm-decomposition-spaces-Mabar} and \eqref{eq:Hb-coincide-avecMabar}, we should prove that $h\in \mathcal{M}(\overline{ {a_i}})$. Observe that 
\[
\ \frac{ p\circ \phi}{\prod_{j=1}^n(1-\overline{\lambda_j}\phi)^{m_j}}\in H^\infty, 
\] 
and if $u(\zeta_i)= u^\prime(\zeta_i)=\cdots=u^{(m_i-1)}(\zeta_i)=0$
then, by \eqref{eq:description-Ma-avec-derivee}, $u\in\mathcal{M}(a_i)$. Hence, $h\in\mathcal M(a_i)\subset\mathcal M(\overline{a_i})$. 

So, we may assume that there is $0\leq \ell_i\leq m_i-1$ such that  \[
u(\zeta_i)= u^\prime(\zeta_i)=\cdots=u^{(\ell_i-1)}(\zeta_i)=0\text{ and }u^{(\ell_i)}(\zeta_i)\ne 0.
\] 
Let $v=u\cdot p\circ \phi$. According to  \Cref{thm5}, we know that $v\in\mathcal{M}(\overline{a_i})$, and \Cref{remark6} implies that $v(\zeta_i)= v^\prime(\zeta_i)=\cdots=v^{(\ell_i-1)}(\zeta_i)=0$. Let us now justify that \[\psi:=\frac{1}{\prod_{j=1}^n(1-\overline{\lambda_j}\phi)^{m_j}}\in\mathcal{M}(\overline{a_{i,\ell_i}}).\] 
It follows from \Cref{Lemma:phi-in-M-ai-li} that $\phi\in\mathcal{M}(\overline{a_{i,\ell_i}})\cap H^\infty $. Applying now \cite{alhajj2022composition}*{Corollary 2.2}
 gives 
\[\frac{ 1}{1-\overline{\lambda_j}\phi}\in \mathfrak{M}( \mathcal{M}(\overline{ {a_{i,\ell_i}}})),\qquad \text{for every }1\leq j\leq n.\]
Since $\mathfrak{M}( \mathcal{M}(\overline{ {a_{i,\ell_i}}})) $ is an algebra, we deduce that $\psi\in \mathfrak{M}( \mathcal{M}(\overline{ {a_{i,\ell_i}}})).$
In particular, we obtain that $\psi\in   \mathcal{M}(\overline{ {a_{i,\ell_i}}} )\cap H^\infty$. We can now apply  \Cref{Lemma4}  to conclude that $h=\psi v\in \mathcal{M}(\overline{ {a_i}})$, which proves \Cref{cor:7}.
\end{proof}
\section{Restriction on non-tangential limits}\label{section:limites}

The goal of this section is to prove that, when one of the assertion in \Cref{thm10} is satisfied, the following condition  holds:
\begin{enumerate}[(NC)]
\item\label{H3} For any $1\leq i\leq n$, we have either $u^{(\ell)}(\zeta_i)=0$ for all $0\leq \ell\leq m_i-1$, or there exists an $\ell$ with $0\leq \ell\leq m_i-1$ such that $u^{(\ell)}(\zeta_i)\not=0$ and then $\varphi$ has a non-tangential limit at $\zeta_i$ that satisfies $\phi(\zeta_i)\in \mathbb{D}\cup Z({a})$.
\end{enumerate}
More specifically, in \Cref{Subsec41}, we prove the necessity of \ref{H3} under the boundedness of $\W$ on $\HH(b)$ and in \Cref{Subsec42}, we prove that \ref{H3} is automatic when $w\in H^2$.
\subsection{Necessity of \ref{H3} under the boundedness of $\W$}\label{Subsec41}

\begin{theorem}\label{lemma8}
Assume that the operator $\W$ is bounded on $\HH(b)$. Then \ref{H3} is satisfied.
%and assume that for some $1\leq i\leq n,$ there exists an $\ell$ with $0\leq \ell\leq m_i-1 $ such that $u^{(\ell)}(\zeta_i)\not=0.$ Then $\phi$ has a non-tangential limit at $\zeta_i$ with $\phi(\zeta_i)\in \mathbb{D}\cup Z({a})$ where $Z({a})=\{\zeta_1,\dots,\zeta_n\}$, meaning that \ref{H3} is satisfied.
\end{theorem}
\begin{proof}
Since $W_{u,\phi}$ is bounded on $\HH(b)$, we know that $u$ and $\varphi$ satisfy \ref{H1} and \ref{H2}. 

Let us now consider $0\leq \ell_i\leq m_i-1$ satisfying
    %Assume that $\W\in\mathcal{L}(\HH(b))$ and that there exists $0\leq \ell\leq m_i-1$ such that  \begin{equation}\label{star}
     %   u^{(\ell)}(\zeta_i)\not=0
      %   \end{equation}
%Let us consider the smallest $0\leq \ell_i\leq m_i-1$ satisfying \Cref{star}, meaning that 
\[
u(\zeta_i) =u^\prime(\zeta_i)=\cdots=u^{(\ell_i-1)}(\zeta_i)=0\text{ and } u^{(\ell_i)}(\zeta_i)\not=0.
\] 
It follows from \Cref{Lemma:phi-in-M-ai-li} that $\varphi$ has a non-tangential limit at $\zeta_i$ satisfying $|\varphi(\zeta_i)|\leq 1$. It remains to prove that if $|\phi(\zeta_i)|=1$, then necessarily $\phi(\zeta_i)\in Z({a})$.
Using \cite{alhajj2022composition}*{Lemma 3.4}, there exists $r_k\to1^-$ such that the sequence $\left( \phi(r_k\zeta_i)\right)_k$ is an interpolating sequence for $H^\infty$.
In particular, there exists $f\in H^\infty$ such that \[f(\phi(r_k\zeta_i))=\begin{cases}
    1&\text{if $k$ is even},\\
    0 &\text{if $k$ is odd}.
\end{cases}\]
Since ${a}H^\infty\subset {a}H^2\subset\HH(b)$, the function $af$ belongs to $\HH(b)$, and therefore the function $\psi=W_{u,\phi}(af)=u\cdot {a}\circ\phi\cdot f\circ \phi$ does as well. In particular, its derivatives up to order $m_i-1$ have non-tangential limits at $\zeta_i$, and according to \Cref{remark6}, we see that $\psi^{(k)}(\zeta_i)=0$ for every $0\leq k\leq\ell_i-1$. So, we have
\[
\frac{\psi^{(\ell_i)} (\zeta_i)}{\ell_i !}=\lim_{k\to+\infty}\frac{\psi(r_k\zeta_i)}{(r_k\zeta_i-\zeta_i)^{\ell_i}}.
\]
But observe that 
\begin{align*}
    \frac{\psi(r_k\zeta_i)}{(r_k\zeta_i-\zeta_i)^{\ell_i}}= \frac{u(r_k\zeta_i)}{(r_k\zeta_i-\zeta_i)^{\ell_i}}\cdot {a}(\phi(r_k\zeta_i))\cdot f(\phi(r_k\zeta_i))=\begin{cases}
    \frac{u(r_k\zeta_i)}{(r_k\zeta_i-\zeta_i)^{\ell_i}}\cdot {a}(\phi(r_k\zeta_i)) &\text{if $k$ is even},\\
    0&\text{if $k$ is odd,}
    \end{cases}
\end{align*}
and 
\begin{align*}
\lim_{k\to\infty} \frac{u(r_k\zeta_i)}{(r_k\zeta_i-\zeta_i)^{\ell_i}}\cdot{a}(\phi(r_k\zeta_i))= u^{(\ell_i)}(\zeta_i)\cdot{a}(\phi(\zeta_i)).
\end{align*}
In order for $\psi^{(\ell_i)}(\zeta_i)$ to exist, we must necessarily have $ u^{(\ell_i)}(\zeta_i){a}(\phi(\zeta_i))=0.$ Since the term $ u^{(\ell_i)}(\zeta_i)$ is not zero, this implies that  ${a}(\phi(\zeta_i))=0$, which concludes the proof of \Cref{lemma8}.
\end{proof}

\subsection{Necessity of \ref{H3} when $w\in H^2$}\label{Subsec42}

\begin{theorem}\label{thm:H3necwH2}
Let $u$ and $\phi$ satisfy \ref{H1} and \ref{H2}, let $w$ be defined by \eqref{eq:defining-w}. Assume that $w\in H^2$. Then \ref{H3} is satisfied.
\end{theorem}

\begin{proof}
Consider $0\leq \ell_i\leq m_i-1$ satisfying
\[
u(\zeta_i) =u^\prime(\zeta_i)=\cdots=u^{(\ell_i-1)}(\zeta_i)=0\text{ and } u^{(\ell_i)}(\zeta_i)\not=0.
\] 
It follows from \Cref{Lemma:phi-in-M-ai-li} that $\varphi$ has a non-tangential limit at $\zeta_i$ satisfying $|\varphi(\zeta_i)|\leq 1$. It remains to prove that $\varphi(\zeta_i)\in\mathbb D\cup Z(a)$.

Since $w\in H^2$ we know that $u\cdot a\circ\phi\cdot p\circ\phi\in aH^2$, where 
\[
p(z)=\prod_{j=p+1}^n (z-\lambda_j)^{m_j}
\]
with $\lambda_j=\phi(\zeta_j)\in\mathbb D$, $p+1\leq j\leq n$. Using \ref{H1}, we have $u\in \mathcal M(\overline{a})\subset \mathcal M(\overline{(z-\zeta_i)^{m_i}})$, and in particular, we may write 
\[
u(z)=(z-\zeta_i)^{m_i}u_1(z)+\sum_{k=\ell_i}^{m_i}\alpha_k (z-\zeta_i)^k,
\]
with $u_1\in H^2$, $\alpha_k\in\mathbb C$, $\ell_i\leq k\leq m_i$ and $\alpha_{\ell_i}=\frac{u^{(\ell_i)}(\zeta_i)}{\ell_i!}\neq 0$. This can be rewritten as 
\[
u(z)=(z-\zeta_i)^{\ell_i}(\alpha_{\ell_i}+(z-\zeta_i)u_2(z)),
\]
where $u_2\in H^2$. Hence
\[
u(z)a(\phi(z))p(\phi(z))=(z-\zeta_i)^{\ell_i}(\alpha_{\ell_i}+(z-\zeta_i)u_2(z)) a(\phi(z))p(\phi(z)).
\]
But, using $u\cdot a\circ\phi\cdot p\circ\phi\in aH^2\subset (z-\zeta_i)^{m_i}H^2$, there exists $u_3\in H^2$ such that
\[
(\alpha_{\ell_i}+(z-\zeta_i)u_2(z))a(\phi(z))p(\phi(z))=(z-\zeta_i)^{m_i-\ell_i}u_3(z).
\]
Thus, 
\begin{equation}\label{eq:qsdqsdqdqsd54343DSDS}
    \alpha_{\ell_i}a(\phi(z))p(\phi(z))=(z-\zeta_i)u_4(z),
\end{equation}
where $u_4(z)=(z-\zeta_i)^{m_i-(\ell_i+1)}u_3(z)-u_2(z)a(\phi(z))p(\phi(z))$. Since $\ell_i\leq m_i-1$, observe that $u_4\in H^2$, whence $(z-\zeta_i)u_4(z)\to 0$ as $z\to\zeta_i$ non tangentially. In particular, if we let $z\to\zeta_i$ non tangentially in \eqref{eq:qsdqsdqdqsd54343DSDS}, we get 
\[
\alpha_{\ell_i}a(\phi(\zeta_i))p(\phi(\zeta_i))=0.
\]
But $\alpha_{\ell_i}\neq 0$. Therefore, either we have $a(\phi(\zeta_i))=0$, which means that $\phi(\zeta_i)\in Z(a)$, or we have $p(\phi(\zeta_i))=0$, which means that there is a j with $p+1\leq j\leq n$ such that $\phi(\zeta_i)=\lambda_j$ and in this case $\phi(\zeta_i)\in \mathbb D$. 
\end{proof}

\section{Proof of \Cref{thm10}}\label{section:proof}
In this section, we will prove our main result. 
Recall that, under conditions \ref{H1} and \ref{H2}, the weight $w$ is given by
\[
w=\frac{u\cdot a\circ\phi\cdot\prod_{j=p+1}^n\left(\phi-\phi(\zeta_j)\right)^{m_j}}{a},
\]
where for every $p+1\leq j\leq n$, $\varphi(\zeta_j)\in\mathbb D$ and $u^{(\ell)}(\zeta_j)\neq 0$, for some $0\leq \ell\leq m_j-1$. 
The first step to prove \Cref{thm10} is to prove that this weight $w$ belongs to $H^2$ whenever $\W$ is bounded on $\HH(b)$. This is given by the following lemma.
\begin{lemma}\label{lemma9} 
Assume that  $u$ and $\varphi$ are such that $W_{u,\varphi}$ is bounded on $\HH(b)$. Then $w\in H^2$.
\end{lemma}
\begin{proof}
According to \Cref{lemma8}, we know that $u$ and $\phi$ necessarily satisfy \ref{H1}, \ref{H2} and \ref{H3}. Then, we can apply \Cref{thm5} which gives that $u\cdot p\circ\phi\in \HH(b)$ for every polynomial $p$. In particular \[v:=u\cdot({a}\circ\phi)\cdot \prod_{j=p+1}^n\left(\phi-\phi(\zeta_j)\right)^{m_j}\in\HH(b).\]
%It follows from \eqref{eq:decomposition-Hb} and \eqref{eq:partie-polynomial-interpolation} that there exists $v_1\in H^2$ such that 
%\[v={a}v_1+\sum_{i=1}^{n} \sum_{k=0}^{m_i-1}v^{(k)}(\zeta_i)r_{i,k},\]
%where $r_{i,k}$ are the Hermite polynomials of degree less or equal to $N-1$ such that 
%\[r_{i,k}^{(\ell)}(\zeta_j)=\begin{cases}
 %   1&\text{ if } i=j\text{ and }k=\ell;\\
  %  0&\text{ otherwise}.
%\end{cases}. \]
According to \eqref{eq:description-Ma-avec-derivee}, in order to prove that $w=\frac{v}{{a}}\in H^2$, it is sufficient to prove that
\begin{equation}\label{eq:6}
    v^{(k)}(\zeta_i)=0, \quad\forall \ 1\leq i\leq n, \ \forall \ 0\leq k\leq m_i-1.
\end{equation}
Notice that, since $v\in\HH(b)\subset \mathcal{M}(\overline{a_i}), $ we can write
\begin{equation}\label{eq:sdqsdqsdsq23443DZD}
v(z)=(z-\zeta_i)^{m_i}v_1+p(z),
    \end{equation}
where $v_1\in H^2$ and $p$ is the polynomial given by
\[
p(z)=\sum_{k=0}^{m_i-1}\frac{v^{(k)}(\zeta_i)}{k!}(z-\zeta_i)^k. 
\]

\par\medskip
\noindent -- \textit{Case 1.} We first prove \eqref{eq:6} for $p+1\leq i\leq n$. Observe that 
\begin{equation}\label{eq:3434DZSDFFD3933}
v=u\cdot (\phi-\phi(\zeta_i))^{m_i}\cdot p_i\circ \phi,
\end{equation}
where $p_i$ is the polynomial given by
\[p_i(z)={a}(z)\prod_{j=p+1,j\ne i}^n(z-\phi(\zeta_j))^{m_j}.
\]
Since $p+1\leq i\leq n$, there exists $0\leq \ell_i\leq m_i-1$ such that 
\begin{equation}\label{dsds2ersdqsd882}
u(\zeta_i)=\cdots= u^{(\ell_i-1)}(\zeta_i)=0\text{  and } u^{(\ell_i)}(\zeta_i)\ne 0.
\end{equation}
According to \Cref{Lemma:phi-in-M-ai-li}, $\phi\in\mathcal{M}(\overline{a_{i,\ell_i}})$, with $a_{i,\ell_i}(z)=(z-\zeta_i)^{m_i-\ell_i}. $ Therefore, we have
\begin{equation}\label{sdsds1ZEZE9SDsdsd}
\phi(z) =a_{i,\ell_i}(z)\phi_1(z)+\sum_{k=0}^{m_i-\ell_i-1}\frac{\phi^{(k)}(\zeta_i)}{k!}(z-\zeta_i)^k, \text{ where }\phi_1\in H^2.
\end{equation}
We now consider two sub-cases.

\noindent $\bullet$ \textit{Case 1(a).} Suppose that $\ell_i\leq m_i-2$. Then $m_i-\ell_i\geq 2$, and  \Cref{sdsds1ZEZE9SDsdsd} can be written as \[\phi(z)=\phi(\zeta_i)+(z-\zeta_i)\phi^\prime(\zeta_i)+(z-\zeta_i)\varepsilon( z ),
\]
where $\varepsilon$ is a function in $H^2$ satisfying $\varepsilon(z)\to 0$ as $z\to\zeta_i$ non-tangentially. Plugging into \eqref{eq:3434DZSDFFD3933}, we get 
\[v(z)=u(z)(z-\zeta_i)^{m_i}(\phi^\prime(\zeta_i)+\varepsilon(z))^{m_i}p_i(\phi(z)).
\] 
Therefore, with \eqref{eq:sdqsdqsdsq23443DZD}, we get   
\[
    u(z)(z-\zeta_i)^{m_i}(\phi^\prime(\zeta_i)+\varepsilon(z))p_i(\phi(z))=(z-\zeta_i)^{m_i}v_1(z)+p(z).
\]
    Dividing by $(z-\zeta_i)^{m_i},$ we see that $p/(z-\zeta_i)^{m_i}\in H^1$. Since ${\deg }(p)\leq m_i-1$, we deduce that $p=0,$ which implies that $v^{(k)}(\zeta_i)=0$ for every $0\leq k\leq m_i-1.$
    \par\smallskip

\noindent $\bullet$ \textit{Case 1(b).} Suppose that $\ell_i= m_i-1$. In this case,  $a_{i,\ell_i}(z)=(z-\zeta_i)$, and by \eqref{sdsds1ZEZE9SDsdsd}, we have
$\phi(z)=(z-\zeta_i)\phi_1(z)+\phi(\zeta_i)$. Moreover, since $u\in\HH(b)\subset \mathcal M(\overline{a_i})$ and taking into account \eqref{dsds2ersdqsd882}, we also have
\[
    u(z)=(z-\zeta_i)^{m_i}u_1(z)+\alpha (z-\zeta_i)^{m_i-1},
\] 
    where $u_1\in H^2$ and $\alpha=\frac{u^{(m_i-1)}(\zeta_i)}{(m_i-1)!}\ne0.$ Then, plugging into \eqref{eq:3434DZSDFFD3933}, we deduce that
   \[
   v(z)= (z-\zeta_i)^{m_i-1}((z-\zeta_i)u_1(z)+\alpha )(z-\zeta_i)^{m_i}\phi_1^{m_i}(z)p_i(\phi(z)).
   \]
Since $(z-\zeta_i)u_1(z)\to 0$ and $(z-\zeta_i)\phi_1(z)\to 0$ as $z\to\zeta_i$ non-tangentially, and since $p_i\circ\phi\in H^\infty$, we see that $v(z)=o((z-\zeta_i)^{m_i-1})$ as $z\to\zeta_i$ non-tangentially. But, we also have that $(z-\zeta_i)^{m_i}v_1(z)=o((z-\zeta_i)^{m_i-1})$, whence \eqref{eq:sdqsdqsdsq23443DZD} implies that $p(z)=o((z-\zeta_i)^{m_i-1})$ as $z\to\zeta_i$ non-tangentially. This implies that $p=0$, and thus $v^{(k)}(\zeta_i)=0$ for every $0\leq k\leq m_i-1$. 

This concludes the proof of \eqref{eq:6} for the first case when $p+1\leq i\leq n$.
\par\medskip

\noindent -- \textit{Case 2.} We now prove \eqref{eq:6} for $1\leq i\leq p$. For this case, we also consider two sub-cases.

\noindent $\bullet$ \textit{Case 2(a).} Suppose that $u^{(\ell)}(\zeta_i)=0$ for all $0\leq \ell\leq m_i-1$. Then, according to \eqref{eq:description-Ma-avec-derivee}, we get that $u\in\mathcal M(a_i)$. Since $({a}\circ\phi)\prod_{j=p+1}^n(\phi-\phi(\zeta_j))^{m_j}\in H^\infty$, we immediately deduce that $v\in\mathcal M(a_i)$, and using once more \eqref{eq:description-Ma-avec-derivee}, we obtain that $v^{(\ell)}(\zeta_i) =0,$ for every $0\leq \ell\leq m_i-1$.\par\smallskip

\noindent $\bullet$ \textit{Case 2(b).}
Suppose that there exists  $0\leq \ell_i\leq m_i-1$  such that 
\[u(\zeta_i) =u^\prime(\zeta_i)=\cdots=u^{(\ell_i-1)}(\zeta_i)=0\text{ and } u^{(\ell_i)}(\zeta_i)\not=0.
\] 
Then, according to \ref{H3}, since $1\leq i\leq p$, we have that $\phi(\zeta_i)\in Z({a})$. To check \eqref{eq:6}, we shall use a similar argument as in the proof of \Cref{lemma8}. First, we have  \[v(\zeta_i)=u(\zeta_i) {a}(\phi (\zeta_i))\prod_{j=p+1}^n(\phi(\zeta_i)-\phi(\zeta_j))^{m_j} =0,\]  because $\phi(\zeta_i)\in Z({a}).$ Now since $|\phi(\zeta_i)|=1,$ there exists $r_k\to1^-$ such that $(\phi(r_k\zeta_i))_k$ is an interpolating sequence for $H^\infty$. In particular, there exists $f\in H^\infty$ such that \[f(\phi(r_k\zeta_i))=\begin{cases}
    1&\text{if $k$ is even},\\
    0 &\text{if $k$ is odd}.
\end{cases}\]
Define \[h(z)=   {a}(z)\prod_{j=p+1}^n(z-\phi(\zeta_j))^{m_j}f(z),\quad z\in\mathbb D.
\]  
Then $h\in aH^\infty\subset\HH(b)$, and since $\W$ is bounded on $\HH(b)$, we get that $u\cdot h\circ\phi$ belongs also to $\HH(b)$.
But \[u\cdot h\circ \phi =u\cdot  {a}\circ \phi \cdot \prod_{j=p+1}^n(\phi-\phi(\zeta_j))^{m_j}\cdot f\circ\phi =v\cdot f\circ \phi.\] Write $\psi:=v \cdot f\circ \phi$. Then $\psi\in \HH(b)$, and the function $\psi$ can be written as
\begin{equation}\label{eq:dsfds72NSD02EZD122}
\psi(z)=\sum_{k=0}^{m_i-1}\frac{\psi^{(k)}(\zeta_i)}{k!}(z-\zeta_i)^k+(z-\zeta_i)^{m_i-1}\varepsilon(z),
\end{equation}
where $\varepsilon(z)\to0$ when $z\to\zeta_i$ non-tangentially. Observe that \[|\psi(r\zeta_i)|\leq |v(r\zeta_i)|\|f\|_\infty.\] 
Since $v(\zeta_i)=0$, we immediately get $\psi(\zeta_i)=0$. 
Assume now that, for some $0\leq\ell\leq m_i-2$, we have $\psi^{(k)}(\zeta_i)=v^{(k)}(\zeta_i)=0$ for every $0\leq k\leq \ell$. Taking into account of \eqref{eq:dsfds72NSD02EZD122}, we have 
\[
\frac{\psi^{(\ell+1)}(\zeta_i)}{(\ell+1)!}=\lim_{k\to\infty}\frac{\psi(r_k\zeta_i)}{(r_k\zeta_i-\zeta_i)^{\ell+1}}=\lim_{k\to\infty}\frac{v(r_k\zeta_i)}{(r_k\zeta_i-\zeta_i)^{\ell+1}}f(\varphi(r_k\zeta_i)).
\]
But 
\[
\frac{v(r_k\zeta_i)}{(r_k\zeta_i-\zeta_i)^{\ell+1}}f(\varphi(r_k\zeta_i))=\begin{cases}
\frac{v(r_k\zeta_i)}{(r_k\zeta_i-\zeta_i)^{\ell+1}}&\text{if $k$ is even,}\\
0&\text{if $k$ is odd,}
\end{cases}
\]
and by \eqref{eq:sdqsdqsdsq23443DZD}, we also have 
\[
\lim_{k\to\infty}\frac{v(r_k\zeta_i)}{(r_k\zeta_i-\zeta_i)^{\ell+1}}=\frac{v^{(\ell+1)}(\zeta_i)}{(\ell+1)!}.
\]
Therefore, in order that $\psi^{(\ell+1)}(\zeta_i)$ to exist, we must necessarily have $\psi^{(\ell+1)}(\zeta_i)=v^{(\ell+1)}(\zeta_i)=0$. Finally, an induction argument shows that $ \psi^{(k)}(\zeta_i)=v^{(k)}(\zeta_i)=0$ for every $0\leq k\leq m_i-1.$ \par\medskip

This ends the proof of \eqref{eq:6} for the second case and so the proof of \Cref{lemma9}.
\end{proof}

\begin{remark}\label{rem:weigth-dansHardy}
If we carefully look at  the proof of \Cref{lemma9}, the boundedness of $W_{u,\varphi}$ on $\HH(b)$ is only used in the subcase 2(b). In the other cases, we only use the hypothesis \ref{H1} and \ref{H2}. In particular, if $u$ and $\phi$ satisfy \ref{H1} and \ref{H2}, and if we assume furthermore that $\varphi(\zeta_i)\in\mathbb D$ for every $1\leq i\leq n$ such that there exists an $\ell$ with $0\leq \ell\leq m_i-1$ and $u^{(\ell)}(\zeta_i)\neq 0$, then the case 2(b) does not appear and we deduce from the proof that $w\in H^2$, without assuming that $\W$ is bounded on $\HH(b)$.
\end{remark}

We are now ready to prove \Cref{thm10}.
 \begin{proof}[Proof of \Cref{thm10}]
The equivalence $  (ii)  \iff (iii)$ follows from \cite{CONTRERAS2001224} (see also \cite{Gallardo-Gutierrez-2010}).   
\par\medskip
Let us prove the implication $(i) \implies (ii)$. Assume that $\W$ is bounded on $\HH(b).$ It follows from \Cref{lemma9} that $w\in H^2$. Let $f\in H^2$ and write  
\[
g(z):=a(z)\cdot\prod_{j=p+1}^n(z-\phi(\zeta_j))^{m_j}f(z),\qquad z\in\mathbb D.
\]
Then $g\in aH^2\subset\HH(b)$, whence
\[\W(g)=u\cdot {a}\circ\phi\cdot\left( \prod_{j=p+1}^n(\phi-\phi(\zeta_j))^{m_j}\right)\cdot f\circ\phi\in   \HH(b).\] 
In particular, there exists $g\in H^2$ and $p\in\P_{N-1}$ such that 
\[u\cdot{a}\circ\phi\cdot\left(\prod_{j=p+1}^n(\phi-\phi(\zeta_j))^{m_j}\right)\cdot f\circ\phi ={a}g+p.\]
Dividing by ${a},$ we get 
\[w\cdot f\circ\phi=g+\frac{p}{{a}}.\]
But $w\in H^2$ and $f\circ \phi\in H^2$ (by the Littlewood subordination principle), whence the function $w\cdot f\circ\phi$ belongs to $H^1$. Thus $p/{a}\in H^1$. Since ${\deg}(p)\leq N-1<{\deg}({a})$, we get $p=0$. Thus, for every $f\in H^2$ the function $w\cdot f\circ \phi$ belongs to $H^2$, which proves $(ii)$.
\par\medskip
Now we prove the implication $(ii)\implies (i)$. Since $u$ and $\varphi$ satisfy \ref{H1} and \ref{H2}, according to \Cref{thm5}, we know that $u\cdot p\circ \phi\in\HH(b)$ for every polynomial $p$. Thus, in particular, $\W(p)\in\HH(b),$ for every $p\in\P_{N-1},$ and since $\HH(b)=aH^2\oplus \P_{N-1}$, it remains to show that 
\begin{equation}\label{eq:7}
\W({a}f)\in\HH(b),\text{ for every } f\in H^2.
\end{equation}
Consider the finite Blaschke product
\[
B(z)=\prod_{j=p+1}^n\left(\frac{z-\lambda_j}{1-\overline{\lambda_j}z}\right)^{m_j},
\]
where we write $\lambda_j=\phi(\zeta_j)\in\mathbb{D}$ for $\ p+1\leq j\leq n$.
Recall that $H^2=BH^2\oplus K_B$ and  
\[K_{B}=\left\{ \frac{p}{ \prod_{j=p+1}^n\left(
1-\overline{\lambda_j}z \right)^{m_j}}, p\in\P_{N_1-1}\right\},\text{ where } N_1=\sum_ {j=p+1}^nm_j.   \]
See for instance \cite{MR3526203}*{Proposition 5.16}. Let $g=p/\prod_{j=p+1}^n\left(
1-\overline{\lambda_j}z \right)^{m_j}\in K_B.$ Then
\[\W(ag)= \frac{u\cdot q\circ\phi}{ ~\prod_{j=p+1}^n\left(
1-\overline{\lambda_j}\phi \right)^{m_j}~},\]
where $q=ap$ is a polynomial. Using \Cref{cor:7}, we deduce that $\W(ag)\in\HH(b)$ for every $g\in K_B$. Thus, in order to prove \eqref{eq:7}, it remains to prove that $\W(aBH^2)\subset\HH(b)$, i.e. 
\begin{equation}
    \label{eq:8}
    u\cdot {a}\circ\phi \cdot B\circ\phi \cdot h\circ\phi\in\HH(b),\text{ for every } h\in H^2.
\end{equation}
Observe that if $h\in H^2$, then \[u\cdot{a}\circ\phi\cdot B\circ\phi\cdot h\circ\phi = u\cdot {a}\circ\phi\cdot\left(  \prod_{j=p+1}^{n}(\phi-\phi(\zeta_j))^{m_j}\right) \cdot h_1\circ\phi ,\]
where \[
h_1=\frac{h}{\prod_{j=p+1}^n(1-\overline{\lambda_j}z)^{m_j}}\in H^2.
\]
Then 
\[u\cdot{a}\circ\phi\cdot B\circ\phi\cdot  h\circ\phi={a}w\cdot h_1\circ\phi={a}\cdot W_{w,\phi}(h_1).\] 
According to $(ii)$, $W_{w,\phi}(h_1)\in H^2$ and then \[u\cdot {a}\circ\phi \cdot B\circ\phi \cdot h\circ\phi\in {a}H^2\subset \HH(b),\] which proves \eqref{eq:8}. Thus we finally deduce that $\W$ is bounded on $\HH(b)$, and this concludes the proof of \Cref{thm10}.
\end{proof}

In the proof of \Cref{thm10}, to show that $W_{u,\phi}(a K_B\oplus \P_{N-1})\subset\HH(b)$, the key was the following: if $f\in a BH^2$ and $g\in H^2$ with $f= a Bg$, then
\begin{equation}\label{eq:89p}
    W_{u,\varphi}f=u\cdot f\circ \varphi=u\cdot a\circ \varphi\cdot B\circ \varphi \cdot g\circ\varphi= a W_{\widetilde w,\varphi}g,
\end{equation}
where $\widetilde w=u\cdot (a\circ \varphi\,\cdot B\circ \varphi)/a$. Moreover, since $\widetilde w$ and $w$ differs by a factor which is invertible in $H^\infty$, the boundedness of $W_{\widetilde w,\varphi}$ on $H^2$ is equivalent to the boundedness of $W_{w,\varphi}$ on $H^2$ and thus to the boundedness of $\W$ on $\HH(b)$.

Now, if $P:\HH(b)\to aBH^2$ denotes the orthonormal projection from $\HH(b)$ onto $aBH^2$, then \eqref{eq:89p} and \eqref{eq:norm-Hb-decomposition} imply that
\[\|W_{\widetilde w,\phi}\|_{\mathcal L(H^2)}=\|\W P\|_{\mathcal L(\HH(b))}\le\|\W\|_{\mathcal L(\HH(b))}.\]
It is natural to ask if we have equivalence between the norm of $W_{\widetilde w,\phi}$ on $H^2$ and the norm of $\W$ on $\HH(b)$, in the sense that there exists a constant $c>0$ which depends only on $b$ such that $\|\W\|_{\mathcal L(\HH(b))}\le c \|W_{\widetilde w,\phi}\|_{\mathcal L(H^2)}$.  However, this is not true in general, as we can see in the following example.

\begin{example}
Let $b(z)=\frac{1+z}2$, $z\in\mathbb D$. Then $\tilde a(z)=\frac{1-z}2$ and $ a(z)=z-1$. Take $u\equiv1$, and let $0<\varepsilon<1/2$ and $\varphi_\varepsilon\in \HH(b)$ be given by 
\[\varphi_\varepsilon(z)=\varepsilon(z-1),~z\in\mathbb D.\] 
Since $\varphi_\varepsilon(1)=0$, we have that $B(z)=z$, and then
  \[\widetilde w_\varepsilon(z)=\frac{(\varphi_\varepsilon(z)-1)\varphi_\varepsilon(z)}{z-1}=\varepsilon(\varepsilon(1-z)-1),~z\in\mathbb D.\]
In particular, we deduce that $\|\tilde w_{\varepsilon}\|_\infty\le\varepsilon(2\varepsilon+1)\le 2\varepsilon$. Moreover, using a well-known estimate on the norm of composition operator on $H^2$ (see \cite{MR1397026}*{Theorem 3.6}), we have
\[\|C_{\varphi_\varepsilon}\|_{\mathcal L(H^2)}\le \frac{1+|\varphi_\varepsilon(0)|}{1-|\varphi_\varepsilon(0)|}=\frac{1+\varepsilon}{1-\varepsilon}\le 3.\]
Therefore, we see that $W_{\widetilde w,\varphi_{\varepsilon}}$ is bounded on $H^2$ and 
\[\|W_{\widetilde w_\varepsilon,\varphi_\varepsilon}\|_{\mathcal L(H^2)}\le\|\widetilde w_{\varepsilon}\|_\infty \|C_{\varphi_\varepsilon}\|_{\mathcal L(H^2)}\le 6\varepsilon.\]
But $W_{1,\varphi_\varepsilon}1=C_{\varphi_\varepsilon}1=1$. So $\|W_{1,\varphi_\varepsilon}\|_{\mathcal L(\HH(b))}\ge1$ and thus
\[\frac{\|W_{1,\varphi_\varepsilon}\|_{\mathcal L(\HH(b))}}{\|W_{\widetilde w_\varepsilon,\varphi_\varepsilon}\|_{\mathcal L(H^2)}}\ge\frac1{6\varepsilon}\longrightarrow \infty~\text{when}~\varepsilon\to0.\]
\end{example}

\section{Necessary or sufficient conditions and examples}\label{section:csq}
\subsection{Sufficient condition for boundedness via the essential range}\label{Subsection:EssentialRange}
Let $\phi:\mathbb D\to\mathbb D$ be analytic and let $b$ be in the closed unit ball of $H^\infty$. Then the set of admissible weights is 
\[
\mathcal M_{\varphi,b}=\{u\in\HH(b):\W\text{ is bounded on }\HH(b)\}.
\]
It can be easily seen that 
\begin{equation}\label{eq:ainfini-inclus-set-of-admissible-poids}
a H^\infty\subset \{u\in\HH(b):u\varphi\in\HH(b)\text{ and }w\in H^\infty\}\subset \mathcal M_{\varphi,b},
\end{equation}
where $w$ is defined by \eqref{eq:defining-w}.
Indeed, the first inclusion is clear, and the second one is a direct consequence of \Cref{thm10}.

\par\smallskip
In Section~\ref{section:admissible}, we shall study more precisely the inclusions in \eqref{eq:ainfini-inclus-set-of-admissible-poids}, and give another description of $\mathcal M_{\phi,b}$ in the case when $\phi$ is a finite Blaschke product.  In our next result, we shall see that it is not necessary that $w$ be bounded everywhere to imply that $\W$ is bounded on $\HH(b)$. Indeed, in \cite{Gallardo-Gutierrez-2010}*{Thm 2.9}, it is proved that if $w\in H^2$ and if there exists $\delta>0$ and $c_\delta>0$ such that $|w|\le c_\delta$ almost everywhere on the set $A_\delta:=\{\zeta\in\mathbb{T}:|\phi(\zeta)|\geq1-\delta\}$, then  $W_{w,\phi}$ is bounded on $H^2$. This result combined with \Cref{thm10} gives the following.

\begin{corollary}\label{Coro11}
Let $u$ and $\varphi$ satisfy \ref{H1} and \ref{H2}, and let $w$ be the function defined by \eqref{eq:defining-w}. Suppose that $w\in H^2$ and that there exists $\delta>0$ such that \[\supess_{z\in A_\delta}|w(z)|<\infty.\]
Then the operator $\W$ is bounded on $\HH(b)$. 
\end{corollary}
We will now give an example of a weight $u$ which is not a multiplier of $\HH(b)$ and a symbol $\phi$ such that the composition by $\phi$ is not bounded on $\HH(b)$, but nevertheless the associated weighted composition operator $\W$ is bounded on $\HH(b)$.
\begin{example}
Let us consider $b(z)=\frac{1+z}2$, $z\in\mathbb D$. Then $\tilde a(z)=\frac{1-z}{2}$, $a(z)=z-1$ and $\HH(b)=(z-1)H^2\oplus\mathbb C.$ Let $\phi:\mathbb D\longrightarrow\mathbb D$ defined by
\[\phi(z)=1-(1-z)^{1/2},~z\in\mathbb D.\]
Then $C_\phi$ is not bounded on $\HH(b)$. Indeed, let $f(z)=1-z$, $z\in\mathbb D$. Then $f\in \HH(b)$ and $C_\phi f(z)=(1-z)^{1/2}$. The function $C_\phi f$ is in the disc algebra $A(\mathbb D)$ (i.e., the space of analytic functions on $\mathbb D$ and continuous on $\overline{\mathbb D}$), and satisfies $C_\phi f(1)=0$. Thus, $C_\phi f\in \HH(b)$ if and only if there exists $h\in H^2$ such that $C_\phi f=(z-1)h$. But this is equivalent to $h=-(1-z)^{-1/2}$, which does not belong to $H^2$. Therefore, $C_\phi f\notin\HH(b)$ and $C_\phi$ is not bounded on $\HH(b)$. Now, let $u$ be defined by
\[u(z)=\frac{(1-z)^{3/4}}{(1+z)^{1/4}},~z\in\mathbb D.\]
Since $|u(z)|\longrightarrow \infty$ when $z\to -1$, the function $u$ is not bounded. But, remind that $\mathfrak{M}(\HH(b))=H^\infty\cap \HH(b)$, and then it follows that the multiplication operator by $u$ is not bounded on $\HH(b)$.
\par\smallskip
Let us now check that we can apply \Cref{Coro11} to get the boundedness of $\W$ on $\HH(b)$. First, note that 
\[u(z)=(z-1)\underset{\in H^2}{\underbrace{\frac{-1}{(1-z)^{1/4}(1+z)^{1/4}}}}\quad\text{and}\quad u(z)\phi(z)=(z-1)\underset{\in H^2}{\underbrace{\frac{(1-z)^{1/2}-1}{(1-z)^{1/4}(1+z)^{1/4}}}},\]
and then $u,\,u\phi\in aH^2\subset\HH(b)$. %On the other hand, $u(1)=0$. 
In particular, $u$ and $\varphi$ satisfy \ref{H1} and \ref{H2}. 
Thus it remains to prove that $\supess_{A_\delta}|w|$ is finite for some $\delta>0$. With our choice of $u$ and $\varphi$, we have 
\[
 w(z)=\frac{u(z)(a\circ\phi)(z)}{a(z)}=\frac{(1-z)^{3/4}\big(1-z)^{1/2}}{(1+z)^{1/4}(1-z)}=\frac{(1-z)^{1/4}}{(1+z)^{1/4}},~z\in\mathbb D.
 \]
Note that $\varphi\in A(\mathbb D)$ and $\phi(-1)= 1-\sqrt2$. Then, by the continuity of $\phi$ on $\mathbb T$, for a sufficient small 
$\delta>0$ and for some open neighborhood $V$ of $-1$, we have $A_\delta\subset \mathbb T\setminus V$. It follows now easily that $w$ is bounded on $\mathbb T\setminus V$, and thus on $A_\delta$. Therefore we can  apply \Cref{Coro11} to conclude that $\W$ is bounded on $\HH(b)$. 
\end{example}
\subsection{Sufficient condition for boundedness via non-unimodular radial limits of $\phi$ at the zeroes of $a$}
\begin{corollary}
Let $u$ and $\phi$ satisfy \ref{H1} and \ref{H2}. Moreover, assume that for every $1\leq i\leq n$, we have 
\[
\limsup_{z\to\zeta_i}|\phi(z)|<1. 
\]
Then, for every $1\leq i\leq n$, there exists an open arc $V_i$ containing $\zeta_i$ such that if $u$ is essentially bounded on $\mathbb T\setminus(\bigcup_{i=1}^n V_i)$, then the operator $\W$ is bounded on $\HH(b)$.
\end{corollary}
The proof of this result is similar to the proof of Corollary 4.12 in \cite{alhajj2022composition}.
\begin{proof}Since for every $1\le i\le n$, $\limsup_{z\to\zeta_i}|\phi(z)|<1$, there exist $\delta>0$ and $0<L<1$ such that for every $1\le i\le n$ and every $z\in\mathbb D$, if $|z-\zeta_i|<\delta$ then $|\phi(z)|\le L$. Let $V_i=\{\zeta\in\mathbb T:|\zeta-\zeta_i|<\delta\}$ and $V=\bigcup_{i=1}^n V_i$. Then $|\phi|\le L$ almost everywhere on $V$. 
Assume now that there exists $1\leq i\leq n$ and $0\leq \ell\leq m_i-1$ such that $u^{(\ell)}(\zeta_i)\neq 0$. Then, according to \Cref{Lemma:phi-in-M-ai-li}, we know that $\phi$ has a non-tangential limit at $\zeta_i$, and it follows from the hypothesis that $\phi(\zeta_i)\in\mathbb D$. In particular, %\ref{H3} is satisfied and 
we deduce from \Cref{rem:weigth-dansHardy} that the weight $w$ belongs to $H^2$. 
\par\medskip
Assume now that $u$ is essentially bounded on $\mathbb T\setminus V$, and let us prove  
\[\sup_{\lambda\in \mathbb{D}} \int_{\mathbb{T}}(1-|\lambda|^2) \frac{|w(\zeta)|^2}{|1-\overline{\lambda}\phi(\zeta)|^2}dm(\zeta)<\infty.
\]
First, remark that for every $\zeta\in\mathbb T\setminus V$, we have $|a(\zeta)|\ge \delta^N$. In particular, since $u$ is essentially bounded on $\mathbb T\setminus V$, we deduce that there exists $C>0$ such that $|w|\le C$ almost everywhere on $\mathbb T\setminus V$. Therefore, for every $\lambda\in\mathbb D$, we obtain
\[\int_{\mathbb{T}\setminus V}(1-|\lambda|^2) \frac{|w(\zeta)|^2}{|1-\overline{\lambda}\phi(\zeta)|^2}dm(\zeta)\le C^2(1-|\lambda|^2)\|C_\phi k_\lambda\|_{2}^2. \]
By the Littlewood subordination principle, the operator $C_\varphi$ is bounded on $H^2$, whence $\|C_\varphi k_\lambda\|^2_{2}\leq \|C_\varphi\|_{\mathcal L(H^2)}^2\|k_\lambda\|^2_2=\|C_\varphi\|_{\mathcal L(H^2)}^2 (1-|\lambda|^2)^{-1}$, which gives
\[
\sup_{\lambda\in\mathbb D}\int_{\mathbb T\setminus V}1-|\lambda|^2) \frac{|w(\zeta)|^2}{|1-\overline{\lambda}\phi(\zeta)|^2}dm(\zeta)\leq C^2\|C_\varphi\|_{\mathcal L(H^2)}^2.
\]
For the integral on $V$, using that $|\varphi|\leq L<1$ a.e. on $V$, we have 
\[\int_{V}(1-|\lambda|^2) \frac{|w(\zeta)|^2}{|1-\overline{\lambda}\phi(\zeta)|^2}dm(\zeta)\le \frac{\|w\|_{H^2}^2}{(1-L)^2}.\]
Thus, we obtain  
\[\sup_{\lambda\in \mathbb{D}} \int_{\mathbb{T}}(1-|\lambda|^2) \frac{|w(\zeta)|^2}{|1-\overline{\lambda}\phi(\zeta)|^2}dm(\zeta)\le C^2\|C_\phi\|_{\mathcal L(H^2)}^2+ \frac{\|w\|_{H^2}^2}{(1-L)^2}<\infty.\]
We conclude now by \Cref{thm10} which implies that $\W$ is bounded on $\HH(b)$. 
\end{proof}
\subsection{Necessary condition for boundedness via Carathéodory derivative}
Recall first that a function $\varphi$ in the closed unit ball of $H^\infty$ has an {\emph{angular derivative in the sense of Carath\'eodory} (briefly an ADC) at the point $\zeta\in\mathbb T$ if $\varphi$ and $\varphi'$ both have a non-tangential limit at $\zeta$ and $|\varphi(\zeta)|=1$. A well-known characterization of Carath\'eodory says that $\varphi$ has an ADC at $\zeta\in\mathbb T$ if and only if 
\[
c:=\liminf_{z\to\zeta}\frac{1-|\varphi(z)|}{1-|z|}<\infty.
\]
Moreover, in this case, $c=|\varphi'(\zeta)|>0$. See for instance \cite{MR1397026}*{Theorem 2.44} or \cite{Fricain-2015-vol2}*{Theorem 21.1}.

We have seen in \Cref{lemma8} that if $\W$ is bounded on $\HH(b)$ and if there exists $1\leq k\leq n$ such that $u(\zeta_k)\neq 0$, then $\varphi$ has a non-tangential limit at $\zeta_k$ with $\varphi(\zeta_k)\in Z(a)\cup\mathbb D$. If $\varphi(\zeta_k)\in Z(a)$, we can say more on the boundary behaviour of $\varphi$.

\begin{corollary}\label{coro:woi}
Let $u$ and $\phi$ such that $\W$ is bounded on $\HH(b)$. Assume that there exists $1\le k,\ell\le n$ such that $\phi(\zeta_k)=\zeta_\ell$. Then 
\[\liminf_{z\to\zeta_k}|u(z)|^2\frac{(1-|\phi(z)|)^{2m_\ell-1}}{(1-|z|)^{2m_k-1}} <\infty.\]
In particular, if $u(\zeta_k)\neq0$ then $m_k\le m_\ell$. If, moreover, $m_\ell=m_k$ then $\phi$ has an angular derivative in the sense of Carathéodory at $\zeta_k$.
\end{corollary}
The proof of this result is similar to the proof of Corollary 4.5 in \cite{alhajj2022composition}.
\begin{proof}
Since $\W$ is bounded on $\HH(b)$, \Cref{thm10} implies that $W_{w,\phi}$ is bounded on $H^2$, and its adjoint $W_{w,\phi}^*$ is bounded as well on $H^2$. It can be easily checked that for $\lambda\in\mathbb D$,  we have $W_{w,\phi}^*k_\lambda=\overline{w(\lambda)}k_{\phi(\lambda)}$. Thus, we deduce that
\[\frac{|w(\lambda)|^2}{1-|\phi(\lambda)|^2}\le \frac{C^2}{1-|\lambda|^2},\qquad\text{for every }\lambda\in\mathbb D,\]
where $C$ is the norm of $W_{w,\varphi}$ on $H^2$. Applying the previous inequality to $\lambda=r\zeta_k$ with $0<r<1$, we obtain 
\begin{equation}\label{eq:540}
    |w(r\zeta_k)|^2\frac{1-r^2}{1-|\phi(r\zeta_k)|^2}\le C^2.
\end{equation}
Observe that
\begin{align*}
w(r\zeta_k)&=u(r\zeta_k)\prod_{j=1}^n\left(\frac{\phi(r\zeta_k)-\zeta_j}{r\zeta_k-\zeta_j}\right)^{m_j}\prod_{j=p+1}^n(\phi(r\zeta_k)-\phi(\zeta_j))^{m_j}\\
&=u(r\zeta_k)\frac{(\phi(r\zeta_k)-\zeta_l)^{m_\ell}}{\big((r-1)\zeta_k\big)^{m_k}}\frac{\displaystyle \prod_{j\neq \ell}(\phi(r\zeta_k)-\zeta_j)^{m_j}}{\displaystyle\prod_{j\neq k}(r\zeta_k-\zeta_j)^{m_j}}\prod_{j=p+1}^n(\phi(r\zeta_k)-\phi(\zeta_j))^{m_j}.
\end{align*}
Then \eqref{eq:540} implies
\begin{equation}
|u(r\zeta_k)|^2\frac{(1-|\phi(r\zeta_k)|)^{2m_\ell-1}}{(1-r)^{2m_k-1}} \le h(r),\label{eq:2093}
\end{equation}
where
\[h(r)=C^2\frac{\displaystyle (1+|\varphi(r\zeta_k)|)\prod_{j\neq k}|r\zeta_k-\zeta_j|^{2m_j}}{\displaystyle (1+r)\prod_{j\neq \ell}|\phi(r\zeta_k)-\zeta_j|^{2m_j}\prod_{j=p+1}^n|\phi(r\zeta_k)-\phi(\zeta_j)|^{2m_j}}.\]
But note that for $j\ge p+1$, $\phi(\zeta_j)$ belongs to $\mathbb D$, whence
\[
A:=\lim_{r\to1}h(r)
=C^2\frac{\displaystyle\prod_{j\neq k}|\zeta_k-\zeta_j|^{2m_j}}{\displaystyle \prod_{j\neq \ell}|\zeta_\ell-\zeta_j|^{2m_j}\prod_{j=p+1}^n|\zeta_\ell-\phi(\zeta_j)|^{2m_j}}<\infty.
\]
We thus deduce that 
\[\liminf_{z\to\zeta_k}|u(z)|^2\frac{(1-|\phi(z)|)^{2m_\ell-1}}{(1-|z|)^{2m_k-1}}\le A <\infty.\]
Suppose now that $u(\zeta_k)\neq0$. Then \eqref{eq:2093} gives
\begin{equation}\label{eq:2398}
\left(\frac{1-|\phi(r\zeta_k)|}{1-r}\right)^{2m_\ell-1}\le h(r)\frac{(1-r)^{2(m_k-m_\ell)}}{|u(r\zeta_k)|^2}.
\end{equation}
If $m_k>m_\ell$, then \eqref{eq:2398}  implies that $\frac{1-|\phi(r\zeta_k)|}{1-r}\longrightarrow0$ when $r\to1$ and we get
\[c=\liminf_{z\to\zeta_k}\frac{1-|\phi(z)|}{1-|z|}=0,
\]
which is not possible by Carathéodory’s theorem (see the discussion at the beginning of this subsection). If $m_k=m_\ell$ then \eqref{eq:2398} gives 
\[c=\liminf_{z\to\zeta_k}\frac{1-|\phi(z)|}{1-|z|}\le\left(\frac{A}{|u(\zeta_k)|^2}\right)^{\frac1{2m_l-1}}<\infty,\]
and thus by Carathéodory’s theorem, the function $\phi$ has an ADC at $\zeta_k$.
\end{proof}
In the particular case where $b(z)=\frac{1+z}{2}$, we will see that the existence of an ADC at $1$ is also a sufficient condition for the existence of a weight $u$ insuring the boundedness of $\W$ on $\HH(b)$.
\begin{corollary}
Let $b(z)=\frac{1+z}2$, $z\in\mathbb D$, and let $\phi:\mathbb D\to\mathbb D$ analytic which has a non-tangential limit at $1$ with $\phi(1)=1$. Then the following assertions are equivalent.
\begin{enumerate}[(i)]
    \item There exists $u\in \HH(b)$ such that $u(1)\neq0$ and $\W$ is bounded on $\HH(b)$.
    \item $\phi$ has an angular derivative in the sense of Carathéodory at $1$.
    \item The operator $C_\phi$ is bounded on $\HH(b)$.
\end{enumerate}
\end{corollary}
\begin{proof}
The implication $(i)\implies(ii)$ is a consequence of \Cref{coro:woi}.

Assuming $(ii)$, i.e., $\phi$ has an angular derivative in the sense of Carathéodory at $1$. Then by \cite{Fricain-2015-vol2}*{Theorem 21.1}, $g=\frac{\phi-1}{z-1}\in H^2$ and thus $\phi=(z-1)g+1\in \HH(b)$. We can now apply Corollary 4.9 in \cite{alhajj2022composition} to get $(iii)$. The implication $(iii)\implies(i)$ is trivial (take $u\equiv 1$ for instance).
\end{proof}
\section{Admissible weights}\label{section:admissible}
In Subsection~\ref{Subsection:EssentialRange}, we defined the set of admissible weights by
\[
\mathcal M_{\varphi,b}=\{u\in\HH(b):\W\text{ is bounded on }\HH(b)\},
\]
where $\phi:\mathbb D\to \mathbb D$ is analytic. We also noticed in \eqref{eq:ainfini-inclus-set-of-admissible-poids} that we always have 
\begin{equation}\label{eq:sdsd233D00JSD}
    aH^\infty\subset\{u\in\HH(b)\,:\,u\phi\in\HH(b)~\text{and}~w\in H^\infty\}\subset\mathcal M_{\phi,b},
\end{equation}
where $w$ is defined by \eqref{eq:defining-w}.

If $\|b\|_\infty<1$, then its Pythagorean mate $\tilde a$ has no zeroes in $\overline{\mathbb D}$ and so $a\equiv 1$, which implies that $\HH(b)=H^2$. In this case, the first inclusion in \eqref{eq:sdsd233D00JSD} is trivially an equality, and we know when the second inclusion  is strict. Indeed, Contreras and Hern\'{a}ndez-D\'{\i}az proved in \cite{MR1983019} that 
\begin{equation}\label{eq:contreras}
 \mathcal M_\phi=H^\infty \Longleftrightarrow \phi \text{ is a finite Blaschke product},
\end{equation}
where $\mathcal M_\phi=\{v\in H^2:W_{v,\phi}\text{ is bounded on }H^2\}$. In this section, we shall show, in particular, that if $b$ is a rational function in the closed unit ball of $H^\infty$ that is not a finite Blaschke product and such that $\|b\|_\infty=1$, then the inclusion $aH^\infty\subset \mathcal M_{\phi,b}$ is always strict.

\subsection{Admissible weights and finite Blaschke products}
In the next result, we show that the second inclusion in \eqref{eq:sdsd233D00JSD} is an equality as soon as $\phi$ is a finite Blaschke product.  
\begin{theorem}\label{Th:fbpoir3u}
Let $b$ be a rational function in the closed unit ball of $H^\infty$ that is not a finite Blaschke product and such that $\|b\|_\infty=1$. Let $\phi$ be a finite Blaschke product. Then     \[M_{\phi,b}=\left\{u\in\HH(b)\,:\,u\phi\in\HH(b)~\text{and}~w\in H^\infty\right\}.\]
\end{theorem}
\begin{proof}

Thanks to \eqref{eq:sdsd233D00JSD}, we always have $\{u\in\HH(b)\,:\,u\phi\in\HH(b)~\text{and}~w\in H^\infty\}\subset\mathcal M_{\phi,b}$, and so we just have to check the other inclusion when $\phi$ is a finite Blaschke product.

  Let $u\in\mathcal M_{\phi,b}$, meaning that $u\in\HH(b)$ and $W_{u,\phi}$ is bounded on $\HH(b)$. In particular, \ref{H1} and \ref{H2} are satisfied and according to \Cref{thm10}, the operator $W_{w,\phi}$ is bounded on $H^2$, where $w=u\cdot a\circ\phi/a$ (note that since $\phi$ is a finite Blaschke product, for every $1\leq i\leq n$, $\phi(\zeta_i)\in\mathbb T$). In particular, $w\in\mathcal M_\phi$ and 
\eqref{eq:contreras} implies that we necessarily have $w\in H^\infty$. Thus
\[
\mathcal M_{\phi,b}\subset \left\{u\in\HH(b)\,:\,u\phi\in\HH(b)~\text{and}~w\in H^\infty\right\}.
\]
Combined with \eqref{eq:sdsd233D00JSD}, this gives the equality and concludes the proof.
\end{proof}
%A natural question is to ask if there is an equivalence in \Cref{Th:fbpoir3u}. In other term we can ask the following question:
We may ask if we have equivalence in \Cref{Th:fbpoir3u}. More precisely:
\begin{question}
Let $\phi:\mathbb D\to \mathbb D$ analytic and suppose that for every rational function $b$ with $\|b\|_\infty=1$ which is not a finite Blaschke product, we have $M_{\phi,b}=\{u\in\HH(b)\,:\,u\phi\in\HH(b)~\text{and}~w\in H^\infty\}$, where $w$ is defined by \eqref{eq:defining-w}. Do we have necessarily that $\phi$ is a finite Blaschke product?
\end{question}
The following result gives a partial answer to this question.
\begin{proposition}
Let $\phi:\mathbb D\to\mathbb D$ analytic and suppose that $\overline{\phi(\mathbb D)}\cap\mathbb T\neq\mathbb T$. Then there exists a rational function $b:\mathbb D\to\mathbb D$ with $\|b\|_\infty=1$ which is not a finite Blaschke product such that
\[\{u\in\HH(b)\,:\,u\phi\in\HH(b)~\text{and}~w\in H^\infty\}\varsubsetneq\mathcal M_{\phi,b}.\]
\end{proposition}
\begin{proof}
Since $\overline{\phi(\mathbb D)}\cap\mathbb T\neq\mathbb T$, then $\phi$ cannot be a finite Blaschke product and then, by \eqref{eq:contreras}, there exists $w\in H^2\setminus H^\infty$ such that $W_{w,\phi}$ is bounded on $H^2$. 

Moreover, by hypothesis, there exists $\zeta\in\mathbb T$ and $\varepsilon>0$ such that $|\phi(z)-\zeta|>\varepsilon$ for every $z\in\mathbb D$. So let $b(z)=\frac{z+\zeta}{2}$, $z\in\mathbb D$. Then we have $a(z)=z-\zeta$, $z\in\mathbb D$, and let 
\[u(z)=\frac{(z-\zeta)w(z)}{\phi(z)-\zeta},~z\in\mathbb D.\]
Since $|\phi-\zeta|>\varepsilon$, we have that $w/(\phi-\zeta)\in H^2$. So we deduce that  $u$ and $u\phi$ belong to 
$ aH^2\subset\HH(b)$ and $u(\zeta)=0$. This implies that $w$ coincide with the weight defined in \eqref{eq:defining-w}. Since the operator $W_{w,\phi}$ is bounded on $H^2$, by \Cref{thm10}, the operator $\W$ is bounded on $\HH(b)$, which means that $u\in\mathcal M_{\phi,b}$. But since $w\notin H^\infty$, this concludes the proof.
\end{proof}

\subsection{Existence of Admissible weights which do not belong in $aH^\infty$} 
The goal of this subsection is to prove that the inclusion $aH^\infty\subset\mathcal M_{\phi,b}$ is always strict. % as soon as $\phi$ is a finite Blaschke product. 

\begin{theorem}\label{Th:InclStrict}
Let $b$ be a rational function in the closed unit ball of $H^\infty$ that is not a finite Blaschke product and such that $\|b\|_\infty=1$.  Then, for every analytic function $\phi:\mathbb D\to\mathbb D$, we have  $aH^\infty\varsubsetneq \mathcal M_{\phi,b}$.
\end{theorem}
To prove this result, we will argue by contradiction. First, we will show that if $\mathcal M_{\phi,b}=aH^\infty$, then $\phi$ must necessarily be a finite Blaschke product (see \Cref{Lemme:CondFBP}). Then, as a consequence of \eqref{eq:sdsd233D00JSD}, we will show that if $\phi$ is a finite Blaschke product, then $\mathcal M_{\phi,b}\neq aH^\infty$ (see \Cref{Coro:cleContr}), giving the desired contradiction. 

\begin{lemma}\label{Lemme:CondFBP}
Let $\phi:\mathbb D\to\mathbb D$ be analytic and suppose that there exists a rational function $b:\mathbb D\to\mathbb D$ which is not inner and such that $\mathcal M_{\phi,b}=a H^\infty$. Then $\phi$ is a finite Blaschke product.
\end{lemma}
\begin{proof}
Assume that $\mathcal M_{\phi,b}=aH^\infty$ and let us prove that $\mathcal M_\phi=H^\infty$. By the Littlewood subordination principle, the inclusion $H^\infty\subset \mathcal M_\phi$ is always true. Let's now take $v\in\mathcal M_\phi$. Then $v\in H^2$ and $W_{v,\phi}$ is bounded on $H^2$. Write $u=av$. Since $u\in\mathcal M(a)$, \eqref{eq:description-Ma-avec-derivee} implies that $u^{(k)}(\zeta_i)=0$ for every $1\leq i\leq n,\,0\leq k\leq m_i-1$. In particular, \ref{H1} and \ref{H2} are satisfied and the associated weight $w$ is given by $w=u\cdot a\circ\phi/a=v\cdot a\circ\phi$. Since $W_{v,\phi}$ is bounded on $H^2$ and $a\circ\phi\in H^\infty=\mathfrak M(H^2)$, we obtain that $W_{w,\phi}$ is bounded on $H^2$ as well. Then \Cref{thm10} implies that $W_{u,\phi}$ is bounded on $\HH(b)$. In other words, $u\in\mathcal M_{\phi,b}=aH^\infty$. Hence $av\in aH^\infty$, meaning that $v\in H^\infty$. We thus have proved that $\mathcal M_\phi=H^\infty$, and it follows from \eqref{eq:contreras} that $\phi$ is a finite Blaschke product. 
\end{proof}
\begin{lemma}\label{Coro:cleContr}
    Let $\phi$ be a finite Blaschke product. Then, for every rational function $b$ in the closed unit ball of $H^\infty$ that is not a finite Blaschke product and such that $\|b\|_\infty=1$, we have $aH^\infty\varsubsetneq \mathcal M_{\phi,b}$.
\end{lemma}
\begin{proof}
Let $b$ be a rational function in the closed unit ball of $H^\infty$ that is not a finite Blaschke product and such that $\|b\|_\infty=1$. Since $\|b\|_\infty=1$, we know that the associated Pythagorean mate $\widetilde{a}$ has at least one zero $\zeta_1\in\mathbb T$. Since $\phi$ is a finite Blaschke product, it maps $\mathbb T$ onto $\mathbb T$, and thus there exists $\zeta\in\mathbb T$ such that $\phi(\zeta)=\zeta_1$. Let $f(z)=(z-\xi)^{-1/3}$, $z\in\mathbb D$, and let $u=af$. Observe that $f\in H^2\setminus H^\infty$. In particular, $u\in aH^2\subset\HH(b)$. Moreover, since $a\circ \phi$ is a rational and analytic function on a neighborhood o $\overline{\mathbb D}$ which vanishes at $\zeta$, there exists a rational function $r$ without poles in $\overline{\mathbb D}$ such that $a\circ\phi(z)=(z-\xi)r(z)$, $z\in\mathbb D$. Therefore
\[
u(z) (a\circ\phi)(z)=a(z)f(z)(z-\zeta)r(z)=a(z)(z-\zeta)^{2/3}r(z),
\]
and then $w\in H^\infty$. According to \eqref{eq:sdsd233D00JSD}, we thus conclude that $u\in\mathcal M_{\phi,b}$ and $u\notin aH^\infty$, which proves that 
$aH^\infty\varsubsetneq \mathcal M_{\phi,b}$.
\end{proof}
\Cref{Th:InclStrict} now follows immediately from \Cref{Lemme:CondFBP} and \Cref{Coro:cleContr}.

\section{Compact composition operators}\label{section:remarks}
Recall that $B$ denotes the finite Blaschke product defined by 
\[
B(z)=\prod_{j=p+1}^n\left(
     \frac{z-\lambda_j}{1-\overline{\lambda_j}z}\right)^{m_j},
\]     
where $\lambda_j=\phi(\zeta_j)\in\mathbb{D}$ for every $\ p+1\leq j\leq n$. 
Then, using \eqref{eq:decomposition-Hb}, the space $\HH(b)$ can be decomposed (via a direct sum) as
  \[\HH(b)= aH^2\oplus\P_{N-1}= a BH^2\oplus a K_B\oplus\P_{N-1},\]
and the subspace $ a K_B\oplus\P_{N-1}$ is of finite dimension. Recall also that \eqref{eq:89p} gives $W_{u,\varphi}(f)=aW_{\widetilde w,\varphi}(g)$ for every $f=aBg$ with $g\in H^2$ and where $\widetilde w$ is the weight defined by $\widetilde w=u\cdot (a\circ \varphi\cdot B\circ \varphi)/a$. 

Now, note that the polynomial $q(z)=\prod_{j=p+1}^n(1-\overline{\lambda_j}z)^{m_j}$ does not vanish on $\overline{\mathbb D}$, and thus the Toeplitz operator $T_q$ is invertible on $H^2$. Moreover, we have 
  \begin{equation}\label{eq=Wtilde-W-equivalent}
  W_{\tilde w,\phi}T_q=W_{w,\phi}.
  \end{equation}
In particular, $W_{w,\phi}$ is bounded (resp. compact/Hilbert--Schmidt) on $H^2$ if and only if $W_{\widetilde w,\phi}$ is. So, thanks to \eqref{eq:89p}, we get the following characterization of the compactness. 
\begin{theorem}
Let $u$ and $\phi$ satisfy \ref{H1} and \ref{H2}. Let $w$ defined by \eqref{eq:defining-w}. Then the following are equivalent:
      \begin{enumerate}[(i)]
          \item $W_{u,\phi}$ is compact on $\HH(b)$;
          \item $W_{w,\phi}$ is compact on $H^2$;
          \item We have  \[\lim_{|\lambda|\to1} \int_{\mathbb{T}}(1-|\lambda|^2) \frac{|w(\zeta)|^2}{|1-\overline{\lambda}\phi(\zeta)|^2}dm(\zeta)=0.\]
      \end{enumerate}
\end{theorem}
\begin{proof}
First note that by \Cref{thm10}, each of the assertions $(i)$, $(ii)$, and $(iii)$ implies the boundedness of $W_{u,\phi}$ on $\HH(b)$. Now, the equivalence $(i)\iff(ii)$ is given by \eqref{eq:89p} and the fact that $\dim(aK_B\oplus\P_{N-1})$ is finite. The equivalence $(ii)\iff(iii)$ follows from \cite{Gallardo-Gutierrez-2010}.
\end{proof}
The characterization to be Hilbert--Schmidt is similar.
\begin{theorem}\label{theoHS}
     Let $u$ and $\phi$ satisfy \ref{H1} and \ref{H2}. Let $w$ defined by \eqref{eq:defining-w}.  Then the following are equivalent:
      \begin{enumerate}[(i)]
          \item $W_{u,\phi}$ is Hilbert–Schmidt on $\HH(b)$;
          \item $W_{w,\phi}$ is Hilbert–Schmidt on $H^2$;
          \item We have \[ \int_{\mathbb{T}} \frac{|w(\zeta)|^2}{1-|\phi(\zeta)|^2}dm(\zeta)<\infty.\]
      \end{enumerate}
\end{theorem}
\begin{proof}
First note that by \Cref{thm10}, each of the assertions $(i)$, $(ii)$, and $(iii)$ implies the boundedness of $W_{u,\phi}$ on $\HH(b)$. Now, the equivalence $(i)\iff(ii)$ is given by \eqref{eq:89p} and the fact that $\dim(aK_B\oplus\P_{N-1})$ is finite. The equivalence $(ii)\iff(iii)$ follows   from the following computation:
\begin{align*}
\sum_{n\ge0}\|W_{w,\phi}(z^n)\|_2^2
&=\sum_{n\ge0}\|w\phi^n\|_2^2\\
&=\sum_{n\ge0}\int_{\mathbb T}|w(\zeta)|^2|\phi(\zeta)|^{2n}\,dm(\zeta)\\
&=\int_{\mathbb T}\frac{|w(\zeta)|^2}{1-|\phi(\zeta)|^2}dm(\zeta).
\end{align*}
But $W_{w,\phi}$ is Hilbert-Schmidt on $H^2$ if and only if $\sum_{n\ge0}\|W_{w,\phi}(z^n)\|_2^2<\infty $, and thus if and only if $\int_{\mathbb{T}} \frac{|w(\zeta)|^2}{1-|\phi(\zeta)|^2}dm(\zeta)<\infty$, which gives the equivalence $(ii)\iff(iii)$.
\end{proof} 
\begin{corollary}
Let $\phi\in H^\infty$ and suppose that $\|\phi\|_\infty<1$. Then, for all $u\in\HH(b)$ such that $u\phi\in \HH(b)$, the operator $W_{u,\phi}$ is Hilbert-Schmidt on $\HH(b)$.
\end{corollary}
\begin{proof}
Let $\delta=\|\phi\|_\infty<1$.  Clearly conditions \ref{H1}, \ref{H2} are satisfied.
%, and \ref{H3} also by \Cref{Lemma:phi-in-M-ai-li}. 
Let now $w$ be defined by \eqref{eq:defining-w}. By \Cref{rem:weigth-dansHardy}, we deduce that $w\in H^2$. Then 
\[\int_{\mathbb{T}} \frac{|w(\zeta)|^2}{1-|\phi(\zeta)|^2}dm(\zeta)\le \frac{\|w\|_2^2}{(1-\delta)^2}<\infty.\]
Thus \Cref{theoHS} implies that $\W$ is Hilbert-Schmidt on $\HH(b)$. 
%In particular, $W_{w,\phi}$ is Hilbert--Schmidt on $H^2$. Then if we combine \Cref{{thm10}} and \Cref{theoHS}, we obtain that $W_{u,\phi}$ is Hilbert-Schmidt on $\HH(b)$.
\end{proof}

\noindent Acknowledgment: This work was supported in part by the project COMOP of the French National Research Agency (grant ANR-24-CE40-0892-01). The authors acknowledge the support of the CDP C$^2$EMPI, as well as of the French State under the France-2030 program, the University of Lille, the Initiative of Excellence of the University of Lille, and the European Metropolis of Lille for their funding and support of the R-CDP-24-004-C2EMPI project. The second author was supported by the Palestinian Quebecer Science Bridge (PQSB), which promotes scientific collaboration in research between Quebec, Canada, and Palestine through the Palestine Academy for Science and Technology and the Fonds de Recherche du Quebec (FRQ), Canada and its three branches; the Fonds de recherche du Quebec - Sante (FRQS), the Fonds de recherche du Quebec - Nature et technologies (FRQNT), and the Fonds de recherche du Quebec - Societe et culture (FRQSC). The third author was supported by the Canada Research Chairs program and the Discovery grant of NSERC (Canada). The fourth author also acknowledges the support of the CNRS.

\bibliographystyle{plain}
 \bibliography{ref}
\end{document}